\documentclass[11pt, reqno]{amsart}
\usepackage[T1]{fontenc}

\usepackage{amsmath}
\usepackage{amssymb}
\usepackage{amsthm}
\usepackage{dsfont}
\usepackage{color}
\usepackage{mathrsfs}
\usepackage{stmaryrd}
\usepackage[all]{xy}
\usepackage[mathcal]{eucal}
\usepackage{verbatim}  %%includes comment environment
\usepackage{fullpage}  %%smaller margins
\usepackage{multicol}
\usepackage{hyperref}
\usepackage{graphicx}
\usepackage{tikz}
\usepackage{tikz-cd}
\usepackage{quiver}
\usepackage{pgfornament}
\usetikzlibrary{matrix,arrows.meta,calc}
\usepackage{upgreek}
\usepackage{etoolbox}
\usepackage[backend=biber,style=alphabetic,maxnames=99,minnames=3]{biblatex}
\def\ZZ{\mathbb{Z}}
\def\QQ{\mathbb{Q}}

\def\NN{\mathbb{N}}

\def\BB{\mathbb{B}}
\def\FF{\mathbb{F}}

\def\RR{\mathbb{R}}
\def\CC{\mathbb{C}}

\def\lcm{\mathrm{lcm}}

\def\ap{\mathrm{Ap}}

\def\C{\mathcal{C}}

\def\o{\overline}
\def\endo{\mathrm{End}}
\def\trop{\NN_{\min}}
\def\rtrop{\RR_{\min}^+}

\def\ideal{\mathrm{Ideal}}

\def\b{\beta}

\def\dec{\mathrm{Dec}}
\def\mmax{m_{\max}}
\def\maxap{\mathrm{Max}}
\def\pf{\mathrm{PF}}
\def\t{\mathrm{t}}

\def\id{\mathrm{Ideal}}
\def\idfg{\id_{fg}}
\def\ns{\mathrm{NS}}
\def\lb{[\![}
\def\rb{]\!]}

\def\b0{\mathbf{0}}
\def\b1{\mathbf{1}}

\theoremstyle{definition}
\newtheorem{theorem}{Theorem}[section]
\newtheorem{proposition}[theorem]{Proposition}
\newtheorem{lemma}[theorem]{Lemma}
\newtheorem{corollary}[theorem]{Corollary}
\newtheorem{definition}[theorem]{Definition}
\newtheorem*{definition*}{Definition}

\newtheorem*{remark}{Remark}
\newtheorem{question}[theorem]{Question}
\newtheorem{example}[theorem]{Example}

\numberwithin{equation}{subsection}
\AtEndEnvironment{example}{\hfill\qedsymbol}

\begin{document}

\title{The Arithmetic of Semirings\\ Part I: Ideals}

\author{Jay Chen}
\address{Dept. of Mathematics and Statistics\\
Vassar College\\
Poughkeepsie, NY 12604\\}
\email{jaychen@vassar.edu}

\author{Trevor Hyde}
\address{Dept. of Mathematics and Statistics\\
Vassar College\\
Poughkeepsie, NY 12604\\}
\email{thyde@vassar.edu}

\author{Dorien Laurens}
\address{Dept. of Mathematics and Statistics\\
Vassar College\\
Poughkeepsie, NY 12604\\}
\email{dlaurens@vassar.edu}

\author{Jasper Piermarini}
\address{Dept. of Mathematics and Statistics\\
UMass Amherst\\
Amherst, MA 01003\\}
\email{jasperpierma@umass.edu}

\author{Harriet Simons}
\address{Dept. of Mathematics and Statistics\\
Vassar College\\
Poughkeepsie, NY 12604\\}
\email{hsimons@vassar.edu}

\begin{abstract}
    We study ideals in the semiring $\NN$ of natural numbers, with a focus on those which are lost when extending from $\NN$ to $\ZZ$.
    This leads to a new perspective on the classical theory of numerical semigroups, including the introduction of a natural multiplicative structure.
    We prove that unique factorization of ideals fails in $\NN$ on several levels, introduce a handful of new tropical multiplicative invariants of numerical semigroups, characterize integral closures of ideals in terms of Newton polygons, and analyze the behavior of classical numerical semigroup invariants with respect to the product operation.
\end{abstract}

\maketitle

\section{Introduction}

Number theory traditionally begins with the semiring $\NN$ of natural numbers: the home of the prime numbers and the fundamental theorem of arithmetic.
Algebraic expedience leads to the introduction of additive inverses and working instead with the ring $\ZZ$ of integers.
We keep extending our numbers from $\ZZ$ to $\QQ$ to $\RR$ and ultimately $\CC$.
And while the virtues of $\CC$ are universally known, at each stage of this extension a tradeoff occurs between useful, interesting structure and simplicity.
Number theory, to a large extent, lives in this space between $\ZZ$ and $\RR$, focused on the prime numbers which get diminished in passing from $\ZZ$ to $\QQ$, and lose their distinction in passing from $\QQ$ to $\RR$.
In this paper and its sequels, we take a step back to study the space between $\NN$ and $\ZZ$.

\begin{question}
\label{question: main}
    What is lost when passing from $\NN$ to $\ZZ$?
\end{question}

We find a wealth of structure hiding in this small gap waiting to be explored.
The integers $\ZZ$ are obtained from $\NN$ by adjoining an additive inverse of the unit $1$.
In this way we may view $\ZZ$ as an additive localization, analogous to the multiplicative localization which gives us $\QQ$ from $\ZZ$.
When localizing a ring, the main structures we lose are ideals and quotients.
For example, all the prime ideals of $\ZZ$ degenerate to the trivial ideal $\langle 1 \rangle$ in the field $\QQ$ and none of the quotient maps $\ZZ \to \FF_p$ extend to $\QQ$.
We begin our investigation of Question~\ref{question: main} by considering the ideals and quotients of $\NN$ which get lost in translation to $\ZZ$.
When localizing rings, these losses are related via the standard correspondence between ideals and quotients.
Without subtraction, however, this correspondence breaks down for semirings like $\NN$, leading to two rich stories.
In this paper we analyze the lost ideals, and in a sequel we study the lost quotients of $\NN$ and their analogs in number fields.

Every ideal in $\ZZ$ is, famously, principal.
The first twist is that this is far from true in $\NN$.
For example, the ideal
\[ 
    A := \langle 2, 3\rangle = \{2a + 3b : a, b \in \NN\} = \{0\} \cup\{n \in \NN : n \geq 2\}
\]
is clearly not principal.
Note that the generators 2 and 3 are coprime, hence generate a trivial ideal in $\ZZ$.
Thus $A$ is a nontrivial ideal in $\NN$ which becomes trivial after the introduction of $-1$.
The ideals in $\NN$ with this property are known as \emph{numerical semigroups} and have been intensively studied for more than a century, although not from this perspective.
\begin{itemize}
    \item Numerical semigroups arise naturally in algebraic geometry as the possible orders of vanishing of polynomial functions along branches at a singular point (see \cite[p.~5]{rosales2009numerical} for a list of references) and as the possible orders of poles of meromorphic functions at smooth points of algebraic curves (see \cite[Thm. 1.6.8]{Stichtenoth2009}).

    \item Semigroup rings $K[A]$ with $A$ a numerical semigroup are an important class of one-dimensional local rings in commutative algebra and are precisely the coordinate rings of one-dimensional affine toric varieties (see \cite[p.~19]{rosales2009numerical} for a list of references).

    \item Numerical semigroups conjecturally correspond to power-sum bases for the field of symmetric functions \cite{Kakeya1925, Kakeya1927}.
\end{itemize}

Our arithmetic perspective on numerical semigroups as ideals in $\NN$ suggests new structures and lines of inquiry.
In particular, the natural multiplicative structure numerical semigroups inherit as ideals appears to have been completely overlooked.
This leads to fundamental questions about factorizations, integrality, and the interaction between ideal multiplication and the many classical invariants defined for numerical semigroups.

\subsection{Results}

Ideals were introduced to arithmetic as a means of correcting the failure of unique factorization in extensions of the integers.
For example, 6 has two incompatible factorizations in the ring $\ZZ[\sqrt{-5}]$, but the ideal $\langle 6 \rangle$ does factor uniquely into prime ideals and in a way that resolves the ambiguity at the level of elements.
It is natural to ask whether this celebrated property of ideals in Dedekind domains extends to ideals in $\NN$.
Unfortunately, this is not the case.
For example, if
\begin{align*}
    A &:= \langle 5, 17\rangle,\\
    B &:= \langle 5, 17, 43\rangle,\\
    C &:= \langle 5, 11, 19, 23 \rangle,
\end{align*}
then one may check that each of these ideals of $\NN$ is irreducible and distinct yet
\(
    AC = BC.
\)
Note that the ideals $A$ and $B$ are quite similar; they agree up till 43.
This sort of failure is common and stems from the fact that the multiplicative semigroup of ideals in $\NN$ is not cancellative.
To see past this issue we consider the universal cancellative quotient $[\ideal(\NN)]$ of the semigroup of ideals.
Elements of $[\ideal(\NN)]$ are equivalence classes $[A]$ of ideals under the similarity relation $A \sim B$ if and only if $AC = BC$ for some nonzero ideal $C$.
This semigroup brings the unique factorization question back to life, but only briefly.

\begin{theorem}
\label{thm: UF fails intro}
    The semigroup $[\ideal(\NN)]$ of similarity classes of ideals in $\NN$ does not have unique factorization into irreducibles.
    Furthermore, the number of irreducible factors is also not well-defined.
\end{theorem}

We prove this result as Theorem~\ref{thm: UF fails}, where we provide two infinite families demonstrating the extent of the failure.
En route to proving Theorem~\ref{thm: UF fails intro} we introduce a number of new multiplicative invariants of ideals and numerical semigroups.
These invariants have a tropical flavor, each capturing some interaction a prime $p$ and the archimedean place.
Let $p \in \NN$ be a prime and let $A \subseteq \NN$ be an ideal.
\begin{itemize}
    \item For each $k \geq 0$ define
    \(
        m_{p,k}(A) := \min\{a\in A : v_p(a) \leq k\}.
    \)

    \item Let $\trop \lb T\rb$ be the semiring of formal power series in the variable $T$ with coefficients in $\trop$, the tropical min-times semiring of natural numbers with $\infty$.
    The additive identity in $\trop \lb T\rb$ is $\sum_{k\geq 0} \infty T^k$ and the multiplicative identity is $1 + \sum_{k\geq 1}\infty T^k$.
    In Lemma~\ref{lemma: dec is semiring} we prove that the subset $\dec \subseteq \trop\lb T\rb$ of series with weakly decreasing coefficient sequences forms a semiring (although with a different multiplicative identity than in $\trop\lb T\rb$).
    Let $S_p : \ideal(\NN) \to \dec$ be defined by
    \[
        S_p(A) := \sum_{k\geq 0}m_{p,k}(A)T^k.
    \]

    \item Let $\lambda \geq 1$ be a real number and let $\rtrop$ be the tropical min-times semiring of nonnegative real numbers with $\infty$.
    Define $\Phi_{p,\lambda} : \ideal(\NN) \to \rtrop$ by
    \[
        \Phi_{p,\lambda}(A) := \min_{k\geq 0} m_{p,k}(A)\lambda^k.
    \]

    \item The \textbf{$p$th Newton polygon} of $A$ is defined by
    \[
        N_p(A) := \text{lower convex hull of } \{(k, \log m_{p,k}(A)) : k \geq 0\}.
    \]
    Given convex sets $X$ and $Y$, let 
    \(
        X + Y = \{x + y : x \in X, y \in Y\}
    \)
    denote their Minkowski sum.
    Let $\C$ denote the semigroup of convex subsets of $\RR^2$ with respect to Minkowski sum.
\end{itemize}
Together with these new invariants, we also recall $m_0(A)$, the smallest positive element of $A$.
The key feature of these invariants is that they respect multiplication of ideals, hence provide essential tools for analyzing this nascent multiplicative structure.

\begin{theorem}
\label{thm: invariants intro}
    Let $p \in \NN$ be prime, let $\lambda \geq 1$ be real, and let $A, B \subseteq \NN$ be ideals.
    Then
    \begin{enumerate}
        \item $m_0(AB) = m_0(A)m_0(B)$,
        \item $m_{p,0}(AB) = m_{p,0}(A)m_{p,0}(B)$,
        \item $S_p(AB) = S_p(A)S_p(B)$,
        \item $\Phi_{p,\lambda}(AB) = \Phi_{p,\lambda}(A)\Phi_{p,\lambda}(B)$,
        \item $N_p(AB) = N_p(A) + N_p(B)$.
    \end{enumerate}
    Furthermore, $m_0, m_{p,0}, \Phi_{p,\lambda}$, and $N_p$ are cancellative invariants, hence are well-defined on similarity classes of ideals.
\end{theorem}

Theorem~\ref{thm: invariants intro} is proved in pieces as Lemma~\ref{lemma: m_p multiplicative}, Lemma~\ref{lemma: S_p hom}, and Proposition~\ref{prop: Newton polygon}.
Note that $\ideal(\NN)$ forms a semiring with respect to addition and multiplication of ideals.
The maps in Theorem~\ref{thm: invariants intro} may all be upgraded to semiring homomorphisms.
Cancellative multiplicative invariants give us insights into the multiplicative structure of $[\ideal(\NN)]$.
For instance, we may use these invariants to prove irreducibility of ideals and similarity classes of ideals.

\begin{proposition}
    Let $p \in \NN$ be prime and let $A \subseteq \NN$ be an ideal.
    \begin{enumerate}
        \item If either $m_0(A)$ or $m_{p,0}(A)$ is prime, then the similarity class $[A]$ is irreducible.
        \item If $A$ is a numerical semigroup and any of the following conditions holds, then $A$ is irreducible.
        \begin{enumerate}
            \item $A$ has two minimal generators,
            \item $A$ has a prime minimal generator,
            \item $m_{p,0}(A)$ and $m_{p,\ell}(A)$ are coprime, where $\ell$ is minimal such that $m_{p,\ell}(A) < m_{p,0}(A)$.
        \end{enumerate}
    \end{enumerate}
\end{proposition}
This result is proved as Lemma~\ref{lemma: irred criteria} and Proposition~\ref{prop: irred criteria}.
To effectively work with the universal cancellative quotient $[\ideal(\NN)]$ one needs a method for testing similarity of ideals.
We show that every similarity class contains a unique maximal element.
If $A \subseteq \NN$ is an ideal, then we say $a \in \NN$ is \textbf{integral} over $A$ if there is some nonzero ideal $B\subseteq \NN$ such that $aB \subseteq AB$.
The \textbf{integral closure} $\o A$ is the set of all natural numbers integral over $A$.

\begin{proposition}
    If $A, B \subseteq \NN$ are ideals, then
    \begin{enumerate}
        \item $\o A$ is an ideal,
        \item $A \sim \o A$, and,
        \item $A \sim B$ if and only if $\o A = \o B$.
    \end{enumerate}
\end{proposition}

See Lemma~\ref{lemma: int closure is ideal} and Proposition~\ref{prop: noetherian implies similar to int closure}.
Thus the problem of testing similarity reduces to the problem of calculating integral closures of ideals, and further simplifications reduce this from ideals to numerical semigroups.
One of our main results is an effective characterization of integral closures in terms of the $p$-Newton polygons of the ideal.

\newpage

\begin{theorem}
\label{thm: integral closure intro}
    Let $A \subseteq \NN$ be an ideal and let $a \in \NN$.
    The following are equivalent
    \begin{enumerate}
        \item $a \in \o A$,
        \item $a^k \in A^k$ for some $k\geq 1$,
        \item $(v_p(a), \log a) \in N_p(A)$ for all primes $p$,
        \item If $A$ is nonzero, then $\Phi_{p,\lambda}(A) \leq a\lambda^{v_p(a)}$ for all primes $p$ and all $\lambda \geq 1$.
    \end{enumerate}
\end{theorem}

This is proved as Lemma~\ref{lemma: support function translation} and Theorem~\ref{thm: integral closure}.
The Newton polygons $N_p(A)$ only impose nontrivial conditions for primes $p$ dividing the multiplicity $m_0(A)$.
Hence Theorem~\ref{thm: integral closure intro} provides an efficient algorithm for computing the integral closure of an ideal in $\NN$.
We implemented this algorithm as well as the calculation of all the invariants introduced above as part of a Python class for ideals in $\NN$.
This class, together with documentation and the code verifying all the computations and examples in this paper may be found at \url{https://github.com/tghyde/numerical-ideals}.

There are many classically defined invariants of numerical semigroups, such as the genus, Frobenius number, Ap\'ery set, and type (see Sections~\ref{sec: NS} and \ref{sec: asymptotics} for definitions).
We study the behavior of these invariants with respect to ideal multiplication, and in particular their asymptotic growth under powers $A^k$.
Let $\mmax(A) := \max_p m_p(A)$.

\begin{theorem}
\label{thm: asymptotic intro}
    Let $A$ be a numerical semigroup, let $F(A)$ and $g(A)$ denote the Frobenius number and genus of $A$, respectively.
    Then
    \[
        \lim_{k\to \infty} \frac{\log F(A^k)}{k} = \lim_{k\to \infty} \frac{\log g(A^k)}{k} = \log \mmax(A).
    \]
\end{theorem}

\noindent This is proved as Theorem~\ref{thm: frob asymptotic}.
This result is a key precursor to our characterization of integral closures (Theorem~\ref{thm: integral closure intro}).
In Section~\ref{sec: symmetry} we study the asymptotic growth of the type invariant along powers $A^k$.
Unlike Theorem~\ref{thm: asymptotic intro}, we were unable to determine a definitive formula or even to prove that $\lim_{k\to\infty} \log \t(A^k)/k$ always exists.
However, for a certain special class of numerical semigroups we determine the growth rate explicitly (see Proposition~\ref{prop: quadratic reduction}).

\subsection{Motivation}

Our primary motivation for this paper and its sequels stems from work of Borger \cite{borger2009lambdaringsfieldoneelement, Borger2016}, Borger and Grinberg \cite{borger2016boolean}, and Borger and Jun \cite{borger2025facets}.
Borger posits that we ought to view positivity as local integrality with respect to the archimedean place.
The kernel of this idea appears, for example, in class field theory as part of the definition of the narrow class group.
Hence by working with rings we miss some archimedean local phenomena.
Borger \cite{Borger2016} illustrates how rich this extra structure may be by constructing extensions of the Witt vector functors from rings to semirings, which make substantive connections with algebraic combinatorics.
Borger and Jun \cite{borger2025facets} develop the categorical foundations of module theory over semirings, highlighting the fact that the formal machinery of algebraic geometry works for semirings and contains strictly more information than its standard counterpart.
While Borger and his collaborators develop algebraic geometry over semirings, our aim is to develop the algebraic number theory of semirings.

For example, every ideal $A \subseteq \NN$ factors as $A = gB$ where $g \in \NN$ and $B$ is a numerical semigroup.
Under the extension map $\NN \to \ZZ$ the ideal $A$ maps to the principal ideal $\langle g \rangle$.
Hence we may view the numerical semigroup $B$ is the archimedean local component of $A$.
This perspective recasts the rich theory of numerical semigroups as the study of ideals in $\NN$ which are purely archimedean.
The hope is that through fleshing out the arithmetic of semirings we may begin to see more parallels between the archimedean place and the finite places.

\subsection{Leiden Disclosure}
We aim to use AI tools following the recommendations of the Leiden Declaration \cite{leiden_declaration_2026}.
We used Google's Gemini and OpenAI's ChatGPT throughout our research process for exploring ideas, writing code for our Python class, assisting with computational experiments, finding connections to the literature, checking our logic, and correcting typos, grammatical issues, and notational inconsistencies.
That said, every word of this paper is written and vetted directly by its authors.
All the mistakes are our own.

\subsection{Acknowledgments}
We gratefully acknowledge the support of the URSI program at Vassar College and the Asprey endowment.
The second author acknowledges support from an AMS-Simons Research Enhancement Grant for Primarily Undergraduate Institution Faculty.
We thank James Borger, Pedro Garc\'ia-S\'anchez, and Nathan Kaplan for valuable feedback on an earlier draft.
Finally, the second author would like to thank Jonathan Gerhard for his sustained interest in and enthusiasm for this project during its long gestation.

\section{Ideals and Integrality in Semirings}
\label{sec: ideals in semirings}

In this paper we use the following standard definitions of a (commutative) semiring.

\begin{definition}
    A \textbf{semiring} $R$ is a set with two associative, commutative binary operations called addition $+$ and multiplication $\cdot$ with identities 0 and 1 respectively, such that multiplication distributes over addition and $a \cdot 0 = 0$ for all $a \in R$.
    We require our semiring homomorphisms to respect both 0 and 1.
\end{definition}

\begin{definition}
    An \textbf{$R$-module} is a commutative monoid $M$ together with a semiring homomorphism $\rho : R \to \endo(M)$.
    Given $r \in R$ and $m \in M$ we write $rm := \rho(r)(m)$.
\end{definition}

\begin{definition}
    An \textbf{ideal} of $R$ is an $R$-submodule $A$ of $R$.
    Equivalently, $A \subseteq R$ is a subset closed under addition such that $ra \in A$ for all $r \in R$ and $a \in A$.
\end{definition}

The standard operations of addition, multiplication, and intersection of ideals are defined the same for ideals in a semiring as they are for ideals in a commutative ring.
Let $\ideal(R)$ denote the set of all ideals of $R$.
With the operations of addition and multiplication of ideals, $\ideal(R)$ forms a semiring with identity elements $\langle 0 \rangle$ and $\langle 1 \rangle$.
Note that $A + A = A$ for all $A \in \ideal(R)$, so $\ideal(R)$ is a $\BB$-algebra where $\BB = \{0,1\}$ is the Boolean semiring.

\subsection{Universal Cancellative Quotient}

Let $G$ be a multiplicative commutative semigroup.
A \textbf{multiplicative subset} $M \subseteq G$ is any subset closed under multiplication.
If $a, b \in G$ and $M$ is a multiplicative subset, then we write $a \sim_M b$ if there exists $c \in M$ such that $ac = bc$.
The relation $\sim_M$ defines a congruence on $G$; let $[G]_M$ denote the quotient of $G$ by $\sim_M$.
The semigroup $[G]_M$ is the universal quotient of $G$ for which elements of $M$ are cancelable.
We apply this construction to the multiplicative semigroup of ideals in a semiring.

\begin{lemma}
    Let $R$ be a semiring and let $M \subseteq \ideal(R)$ be a multiplicative subset of ideals.
    Then $[\ideal(R)]_M$ forms a semiring with operations defined representativewise.
\end{lemma}

\begin{proof}
    All we need to check is that $\sim_M$ respects addition.
    %We already know that similarity respects multiplication, so all that remains is to show that it respects ideal addition.
    Suppose that $A_1 \sim_M A_2$ and $B_1 \sim_M B_2$ with $C, D \in M$ such that $A_1 C = A_2 C$ and $B_1 D = B_2 D$.
    Then
    \[
        (A_1 + B_1)CD = A_1CD + B_1CD 
        = (A_1C)D + (B_1D)C
        = (A_2C)D + (B_2D)C
        = (A_2 + B_2)CD.
    \]
    Hence $A_1 + B_1 \sim_M A_2 + B_2$.
\end{proof}

\subsection{Integral Closure}
\label{sec: integral closure}

Let $A$ be an ideal in a semiring $R$.
We say $A$ is a \textbf{zero divisor} if there is a nonzero ideal $B$ such that $AB = \langle 0 \rangle$.
Note that $A$ is a zero divisor if and only if $A$ has a nonzero annihilator.
Let $\idfg(R) \subseteq \id(R)$ denote the multiplicative subset of finitely generated ideals which are not zero divisors.

\begin{definition}
\label{def: ideal integral}
    We say $r \in R$ is \textbf{integral} over $A$ if there exists a $B \in \idfg(R)$ such that $rB \subseteq AB$.
    The \textbf{integral closure} of $A$ is $\o A := \{r \in R : r \text{ is integral over $A$}\}$.
\end{definition}

For ideals in a ring, Definition~\ref{def: ideal integral} is one of several equivalent definitions of what it means for an element to be integral over an ideal.
For general semirings, these equivalencies fracture into a series of strict implications.
The various shades of integrality may all have their virtues, but Definition~\ref{def: ideal integral} emerges as the right notion for the purpose of studying factorization of ideals in $\NN$.
In a sequel we will study these variations on integrality in detail.

\begin{lemma}
\label{lemma: int closure is ideal}
    If $A \subseteq R$ is an ideal, then $\o A$ is an ideal.
\end{lemma}

\begin{proof}
    Suppose $r_1, r_2 \in \o A$ and let $B_1, B_2 \in \idfg(R)$ be such that $r_i B_i \subseteq A B_i$.
    Then $B_1B_2 \in \idfg(R)$ and
    \[
        (r_1 + r_2)B_1B_2 \subseteq r_1 B_1B_2 + r_2 B_1B_2 
        \subseteq (AB_1)B_2 + (AB_2)B_1 = AB_1B_2.
    \]
    Hence $r_1 + r_2 \in \o A$.
    Suppose $s \in \o A$ and $B \in \idfg(R)$ is an ideal such that $sB \subseteq AB$.
    If $r \in R$, then 
    \[
        rsB \subseteq rAB \subseteq AB,
    \]
    since $A$ is an ideal.
    Thus $rs \in \o A$ and $\o A$ is an ideal of $R$.
\end{proof}

From now on we write $A \sim B$ as a shorthand for $A \sim_{\idfg(R)} B$.
That is, $A \sim B$ if and only if there is some $C \in \idfg(R)$ such that $AC = BC$.
We simply refer to the relation $\sim$ as \textbf{similarity} of ideals.

\begin{lemma}
\label{lemma: int closure props}
	Let $R$ be a semiring, let $A, B \subseteq R$ be ideals, and let $r \in R$.
	\begin{enumerate}
		\item If $A \sim B$, then $\o A = \o B$.
		\item $A + \langle r \rangle \sim A$ if and only if $r \in \o A$.
		\item If $A \subseteq B$, then $\o A \subseteq \o B$.	
	\end{enumerate}
\end{lemma}

\begin{proof}
	(1) %\label{lemma: similarity implies same integral closure}
	Let $C \in \idfg(R)$ be such that $AC = BC$.
    	If $r \in \o A$, then there is some $D \in \idfg(R)$ such that $rD \subseteq AD$.
    	Hence $rCD \subseteq ACD = BCD$.
    	Therefore $r \in \o B$, which proves that $\o A \subseteq \o B$.
    	Thus by symmetry we conclude that $\o A = \o B$.
	
	(2) %\label{lemma: similarity condition}
	Suppose that $B \in \idfg(R)$ is an ideal such that 
   	\[  
        		AB = (A + \langle r\rangle)B = AB + rB.
    	\]
    	This is equivalent to $rB \subseteq AB$, which is precisely what it means for $r \in \o A$.
    	Thus $A + \langle r \rangle \sim A$ if and only if $r \in \o A$.

	(3)
	Suppose $r \in \o A$ and let $C \in \idfg(R)$ be an ideal such that $rC \subseteq AC$.
        Then $A \subseteq B$ implies that $AC \subseteq BC$, hence $rC \subseteq BC$.
        Therefore $r \in \o B$.
\end{proof}

We say a semiring $R$ is \textbf{Noetherian} if every ascending chain of ideals eventually stabilizes.

\begin{proposition}
\label{prop: noetherian implies similar to int closure}
    If $R$ is Noetherian, then $A \sim \o A$ for all ideals $A \subseteq R$.
    In particular, $A \sim B$ if and only if $\o A = \o B$.
\end{proposition}

\begin{proof}
    Since $R$ is Noetherian, there are finitely many elements $b_1, \ldots, b_n \in \o A$ such that $A + \langle b_1, \ldots, b_n\rangle = \o A$.
    Let $A_0 := A$ and then recursively define $A_{i+1} := A_i + \langle b_{i+1}\rangle$.
    Then Lemma~\ref{lemma: int closure props} implies that $A_i \sim A_{i+1}$ for each $i$.
    Therefore $A \sim \o A$.
    This together with Lemma~\ref{lemma: int closure props}(1) proves the final claim.
\end{proof}

Proposition~\ref{prop: noetherian implies similar to int closure} implies that if $R$ is Noetherian, every similarity class in $\id(R)$ contains a maximal element, namely the common integral closure of all the elements in the similarity class.
Thus in this case we can, in principle, determine similarity between ideals by checking if their integral closures are equal.
The following proposition is a useful sufficient condition for integrality.

\begin{proposition}
\label{prop: power certificate}
	Let $A \in \idfg(R)$ and let $r \in R$.
    	If there is some $k \geq 1$ such that $r^k \in A^k$, then $r$ is integral over $A$.
\end{proposition}

\begin{proof}
    Suppose that $r^k \in A^k$.
    Then
    \[
        (A + \langle r\rangle)^k = A^k + rA^{k-1} + \ldots + r^{k-1}A + \langle r^k\rangle
        \subseteq A^k + rA^{k-1} + \ldots + r^{k-1}A
        = A(A + \langle r\rangle)^{k-1}.
    \]
    Our assumption that $A \in \idfg(R)$ implies $(A + \langle r\rangle)^{k-1} \in \idfg(R)$.
    Thus $A + \langle r \rangle \sim A$, which by Lemma~\ref{lemma: int closure props} implies $r$ is integral over $A$.
\end{proof}

In Theorem~\ref{thm: integral closure} we show that Proposition~\ref{prop: power certificate} characterizes integral closures of ideals in $\NN$, which leads to an algorithm for calculating integral closures.
Calculating integral closures in a general semiring seems to be a challenging problem.
Given an ideal $A$ and an element $r \in R$, in order to show that $r \in \o A$ one needs to construct an ideal $B \in \idfg(R)$ such that $rB \subseteq AB$.
This is essentially equivalent to constructing a matrix with entries in $A$ which has $r$ as an eigenvalue.
If such an ideal exists, we may construct it using iterated colon ideals.

\begin{definition}
    If $A, B \subseteq R$ are ideals, then the \textbf{colon ideal} $(A : B)$ is defined by
    \[
        (A : B) := \{r \in R : rB \subseteq A\}.
    \]
\end{definition}

Now let $A \in \id(R)$ and let $r \in R$ be a candidate for $\o A$.
Consider the recursive sequence of ideals defined by $B_0 = R$ and
\[
    B_{k+1} = (AB_k : \langle r\rangle) = \{b \in R : rb \in AB_k\}.
\]
The sequence $B_k$ is monotone decreasing, $B_k \supseteq B_{k+1}$ for all $k\geq 0$.

\begin{lemma}
\label{lemma: colon ideal}
    With notation as above, if there exists $B \in \idfg(R)$ such that $rB \subseteq AB$, then $B \subseteq \bigcap_{k\geq 0} B_k$.
\end{lemma}

\begin{proof}
    We proceed by induction.
    For the base case, clearly $B \subseteq R = B_0$.
    Now suppose that $B \subseteq B_k$.
    Then $rB \subseteq AB \subseteq AB_k$, which implies that $B \subseteq B_{k+1}$.
    Thus $B \subseteq \bigcap_{k\geq 0} B_k$.
\end{proof}

If there are only finitely many ideals between any given pair $B \subseteq C$ and if such a witness exists, then this lemma implies the sequence $B_k$ must converge to a fixed point after finitely many steps, and that fixed point is the maximal witness.
However, it is not clear how to bound the number of iterations required to reach stabilization.
Thus this only gives us a semi-decision procedure for checking if elements belong to the integral closure.

\section{Ideals in the Natural Numbers}
\label{sec: ideals in N}

We now turn from general theory to the special case of $R = \NN$, the initial semiring.
Here we may leverage the well-ordering and arithmetic structure of $\NN$ to construct refined invariants for studying similarity classes of ideals.

\subsection{Numerical Semigroups}
\label{sec: NS}

Every ideal in the ring $\ZZ$ is principal.
However the semiring $\NN$ has non-principal ideal which get lost in translation when passing from $\NN$ to $\ZZ$.
These ideals, known as \emph{numerical semigroups}, have been extensively studied for their many interesting properties and connections to geometry, algebra, and combinatorics.
We briefly review the fundamentals we require.

\begin{definition}
    A \textbf{numerical semigroup} is an ideal $A \subseteq \NN$ which contains a pair of coprime integers.
\end{definition}

Our definition differs slightly from the standard definition of a numerical semigroup.
We adopt this definition with an eye towards generalizations to be explored in a sequel to this paper.
The following well-known result proves equivalence between our definition and the more standard one.

\newpage

\begin{proposition}
\label{prop: mcnugget}
    If $A$ is a numerical semigroup, then $\NN \setminus A$ is finite.
\end{proposition}

\begin{proof}
    Let $A$ be a numerical semigroup and suppose that $a, b \in A$ are coprime elements.
    Then $a, b$ generate the trivial ideal in $\ZZ$, which implies there are natural numbers $m, n$ such that $ma = 1 + nb$.
    If $0 \leq r < b$, then $rma = r + rnb \equiv r \bmod b$ is an element of $A$ which is congruent to $r$ modulo $b$.
    Furthermore, since $b \in A$ all natural numbers larger than $rma$ which are congruent to $r \bmod b$ will also belong to $A$.
    Thus there are finitely many natural numbers which do not belong to $A$.
\end{proof}

Numerical semigroups always have a generating set of coprime elements.

\begin{lemma}
\label{lemma: recognize NS}
    Let $A$ be a nonzero ideal in $\NN$.
    Then $A$ is a numerical semigroup if and only if $A = \langle a_1, \ldots, a_n\rangle$ for some $a_i$ such that $\gcd(a_1,\ldots, a_n) = 1$.
\end{lemma}

\begin{proof}
    First suppose that $A$ is a numerical semigroup.
    Let $a_1 < a_2 < \ldots$ be an enumeration of the nonzero elements of $A$ in order.
    Let $A_k := \langle a_1, \ldots, a_k\rangle$.
    Then $A_k \subseteq A_{k+1}$ and $A = \bigcup_{k\geq 1} A_k$.
    Since $A$ is a numerical semigroup, there must be some $m$ such that $\gcd(a_1,\ldots, a_m) = 1$.
    Hence $A_m$ is a numerical semigroup.
    Proposition~\ref{prop: mcnugget} implies that $A \setminus A_m$ is finite.
    Therefore $A = A_n$ for some $n \geq m$.

    Next suppose that $A = \langle a_1, \ldots, a_n\rangle$ for some $a_i$ such that $\gcd(a_1,\ldots, a_n) = 1$.
    Since $a_1, \ldots, a_n$ generate the trivial ideal in $\ZZ$, it follows that, after possibly reordering the $a_i$, there are natural numbers $m_i$ and some $1 \leq \ell \leq n$ such that
    \[
        m_1a_1 + \ldots + m_\ell a_\ell = 1 + m_{\ell+1}a_{\ell+1} + \ldots + m_na_n.
    \]
    Note that $m_1a_1 + \ldots + m_\ell a_\ell$ and $m_{\ell+1}a_{\ell+1} + \ldots + m_na_n$ are both elements of $A$, and the above identity implies that they are coprime.
    Hence $A$ is a numerical semigroup.
\end{proof}

Every ideal in $\NN$ canonically factors as a principal ideal times a numerical semigroup.

\begin{lemma}
\label{lemma: N ideal is principal times NS}
    If $A \subseteq \NN$ is an ideal, then $A = gB$ for some $g \in \NN$ and numerical semigroup $B$.
\end{lemma}

\begin{proof}
    If $A = \{0\}$, then $g := 0$ and $B = \NN$ will do.
    Now suppose $A$ is nonzero and let $g := \gcd(A)$.
    Let $B := \{a/g : a \in A\} \subseteq \NN$ and note that $B$ is an ideal of $\NN$.
    Let $b_1 < b_2 < \ldots$ be an enumeration of $B$ and define $g_n := \gcd(b_1,\ldots, b_n)$.
    Then $g_n$ is a weakly decreasing sequence of natural numbers which converges to $1$.
    Hence there is some $n$ such that $g_n = 1$. 
    Thus $\langle b_1, \ldots, b_n\rangle$ is a numerical semigroup by Lemma~\ref{lemma: recognize NS}.
    Therefore $B$ itself is a numerical semigroup.
    The identity $A = gB$ follows by construction.
\end{proof}

Let $A$ be a numerical semigroup.
The \textbf{Frobenius number} of $A$, denoted $F(A)$, is the largest element of $\NN \setminus A$.
The \textbf{genus} of $A$ is the cardinality of $\NN \setminus A$.
Every numerical semigroup has a unique minimal set of generators.
The \textbf{embedding dimension} of $A$ is the number of minimal generators of $A$.
For a general overview of numerical semigroups and many references to the literature, we recommend the book by Rosales and Garc\'ia-S\'anchez \cite{rosales2009numerical}.

Let $\ns$ denote the set of all numerical semigroups together with the zero ideal $\langle 0 \rangle$.
Lemma~\ref{lemma: recognize NS} together with Lemma~\ref{lemma: N ideal is principal times NS} implies that $\NN$ is a Noetherian semiring.
Every nonzero element of $\NN$ is regular, hence every nonzero ideal in $\NN$ is regular and finitely generated.
Thus $\id(\NN) = \ideal(\NN)$.
Furthermore, Lemma~\ref{lemma: recognize NS} implies that $\ns$ is closed under addition and multiplication of ideals.
Thus $\ns \subseteq \ideal(\NN)$ is a subsemiring.

While numerical semigroups have been intensely studied, their natural multiplicative structure appears to have been completely overlooked.
The multiplicative structure of ideals is of central interest from a number theoretic perspective.
One natural question to ask is whether or not ideals in $\NN$ factor uniquely into irreducibles.
We show in Theorem~\ref{thm: UF fails} that unique factorization fails at several levels for numerical semigroups. 
That leaves open the problem of analyzing the multiplicative relations in $\ns$.

\subsection{Multiplicities}
\label{sec: multiplicities}

If $A \subseteq \NN$ is a nonzero ideal, then the \textbf{multiplicity} $m_0(A)$ is the smallest positive element in $A$.
We define $m_0(\{0\}) := \infty$.
Clearly we have $m_0(AB) = m_0(A)m_0(B)$ for all ideals $A, B$.
To help us analyze the multiplicative structure of ideals in $\NN$ we construct variations on the multiplicity associated to prime ideals in $\NN$. 
Prime ideals are defined identically for semirings as in rings.
First we classify the prime ideals in $\NN$.

\begin{lemma}
    If $P \subseteq \NN$ is a prime ideal, then either
    \begin{enumerate}
        \item $P = \langle p \rangle$ for $p$ a prime number,
        \item $P = \langle 0\rangle$, or
        \item $P = \langle 2, 3\rangle$.
    \end{enumerate}  
\end{lemma}

\begin{proof}
    Clearly $\langle p \rangle$ is a prime ideal when $p$ is prime or $p = 0$.
    Let $P = \langle 2, 3\rangle = \{0\} \cup \{n \in \NN : n \geq 2\}$.
    If $a, b \in \NN$ are positive integers such that $ab \in P$, then $ab \geq 2$, which implies that either $a \geq 2$ or $b\geq 2$.
    Thus $P$ is prime.
    
    Conversely, suppose $P \subseteq \NN$ is a nonzero prime ideal.
    Lemma~\ref{lemma: N ideal is principal times NS} implies that $P = gB$ where $g := \gcd(P)$ and $B$ is a numerical semigroup.
    Let $m$ be the multiplicity of $B$.
    Then $gm \in P$, which implies that either $g \in P$ or $m \in P$.
    Since $gm$ is the smallest positive element of $gB$, the only way $g \in P$ is if $m = 1$.
    In that case $B = \NN$ and $P = \langle p \rangle$.
    If $m \in P$, then $g = 1$ and $P = B$.
    Proposition~\ref{prop: mcnugget} implies that $P$ is cofinite.
    Hence $p^k \in P$ for every prime $p$ and all sufficiently large $k$.
    Thus $p \in P$ for all primes $p$.
    In particular, $2, 3 \in P$ which implies that $\langle 2, 3\rangle \subseteq P$.
    Since $P$ is a proper ideal, we must have $P = \langle 2, 3\rangle$.
\end{proof}

\begin{remark}
    There are many examples in number theory of objects which are parametrized by prime numbers with one exception.
    The ur-example of this phenomenon is equivalence classes of absolute values on $\QQ$.
    These patterns led to the idea of a \emph{missing prime} which completes our classification.
    Motivated by a geometric analogy, this missing prime is often referred to as the \emph{prime at infinity}.
    Many instances of this phenomenon may be linked back to absolute values.
    The concept of \emph{places} was created in an attempt to put all the primes, including the prime at infinity, on an equal footing.
    
    However, the analogy is incomplete.
    Thus it is notable that $\NN$ contains exactly one prime ideal not visible in $\ZZ$.
    In the sequel we study quotients of arithmetic semirings generalizing $\NN$ to higher degree extensions.
    Although the link between ideals and quotients fractures for semirings, it is also the case that $\NN$ has one more semifield quotient than $\ZZ$.
    \qed
\end{remark}

Let $P \subseteq \NN$ be a prime ideal and let $A$ be an ideal.
We define the \textbf{prime-to-$P$ multiplicity} $m_P(A) := \min(A\setminus P)$, where we set $\min(\emptyset) := \infty$.

\begin{lemma}
\label{lemma: m_p multiplicative}
    If $A, B \subseteq \NN$ are ideals and $P \subseteq \NN$ is a prime ideal, then
    \[
        m_P(AB) = m_P(A)m_P(B).
    \]
\end{lemma}

\begin{proof}
    Let $a_P := m_P(A)$ and $b_P := m_P(B)$.
    Suppose at least one of $a_P$ or $b_P$ is $\infty$; without loss of generality suppose $a_P = \infty$.
    Then every element of $A$ belongs to $P$, hence the same is true for $AB$.
    Thus $m_P(AB) = \infty = m_P(A)m_P(B)$.

    Now suppose that $a_P, b_P < \infty$.
    Then $a_P b_P \in AB \setminus P$ by the definition of a prime ideal.
    Let $m_P(AB) = a_1b_1 + \ldots + a_nb_n$.
    If a sum belongs to the complement of an ideal, then at least one summand must be in that complement, hence $a_ib_i \in AB \setminus P$ for some $i$.
    Hence $a_i, b_i \notin P$.
    Since $a_i \geq a_P$ and $b_i \geq b_P$ we have $m_P(AB) =a_ib_i \geq a_Pb_P$.
    Thus $m_P(AB) = a_Pb_P$.
\end{proof}

If $P = \langle p\rangle$, then we write $m_p$ as a shorthand for $m_{\langle p \rangle}$.
This is consistent with our notation $m_0(A)$ for the multiplicity of $A$.
If $P = \langle 2, 3\rangle$, then 
\[
    m_P(A) = 
    \begin{cases}
        1 & \text{if }A = \NN,\\
        \infty & \text{otherwise.}
    \end{cases}
\]
Hence this invariant just captures triviality.

Suppose $A, B$ are numerical semigroups.
Then $m_p(A), m_p(B) < \infty$.
If $A \sim B$, then there is some ideal $C$ such that $AC = BC$.
Thus 
\[
    m_p(A)m_p(C) = m_p(AC) = m_p(BC) = m_p(B)m_p(C).
\]
Positive integers are multiplicatively cancellative, hence $m_p(A) = m_p(B)$.
Therefore the functions $m_p$ are invariants of similarity classes.

Let $\trop$ be the semiring with underlying set $\NN \cup \{\infty\}$ and operations $a \oplus b := \min(a,b)$ and $a \otimes b := ab$.
We call $\trop$ the \textbf{tropical min-times semiring}.
Note that $\infty$ is the additive identity and $1$ is the multiplicative identity in $\trop$.
Observe that each $m_p$ defines a semiring homomorphism from $\ideal(\NN)$ to $\trop$.

The $m_p$ invariants may be usefully extended to sequence $m_{p,k}$ of invariant functions.
Given a prime $p \in \NN$ let $v_p$ denote the $p$-adic valuation.
If $A$ is an ideal, let
\[
    m_{p,k}(A) := \min \{a \in A : v_p(a) \leq k\}.
\]
Note that $m_{p,0}(A) = m_p(A)$.
The values of $m_{p,k}(A)$ may in general be infinite, but if $A$ is a numerical semigroup, then $m_{p,k}(A) < \infty$.

The $m_{p,k}$ invariants always take values among the generators of an ideal.

\begin{lemma}
\label{lemma: m_pk among generators}
    If $A = \langle a_1, \ldots, a_n\rangle$ is a numerical semigroup, then for each $k$ we have $m_{p,k}(A) = a_i$ for some $i$.
\end{lemma}

\begin{proof}
    Let $a := m_{p,k}(A)$.
    Then $a \in A$ implies there are some $b_i \in \NN$ such that 
    \[
        a = a_1b_1 + \ldots + a_nb_n.
    \]
    The ultrametric inequality implies that
    \[
        k \geq v_p(a) \geq \min_i \{v_p(a_i) + v_p(b_i)\}.
    \]
    Hence there is some $i$ such that $b_i \neq 0$ and $v_p(a_i) \leq k$.
    Since $a_i \leq a$, the minimality of $a$ implies that $a_i = a = m_{p,k}(A)$.
\end{proof}

When $k > 0$ the invariants $m_{p,k}$ are not multiplicative, but rather satisfy a min-times convolution identity.

\begin{lemma}
\label{lemma: min-times convolution}
    Let $p$ be a prime and let $A, B \subseteq \NN$ be ideals.
    Then
    \[
        m_{p,k}(AB) = \min_{i + j = k}\{m_{p,i}(A)m_{p,j}(B)\}.
    \]
\end{lemma}

\begin{proof}
    Suppose that $c := a_1b_1 + \ldots + a_nb_n \in AB$ has $v_p(c) \leq k$.
    Then the ultrametric inequality implies that $v_p(a_ib_i) \leq k$ for some $i$.
    Hence $\min \{c \in AB : v_p(c) \leq k\}$ is achieved by some product $ab$.
    In particular,
    \begin{align*}
        m_{p,k}(AB) &= \min \{ab : a \in A, b \in B, v_p(ab) \leq k\}\\
        &= \min_{i + j \leq k}\{m_{p,i}(A)m_{p,j}(B)\}\\
        &= \min_{i + j = k}\{m_{p,i}(A)m_{p,j}(B)\},
    \end{align*}
    where the last equality follows from the fact that $m_{p,k}(A)$ is weakly decreasing with $k$.
\end{proof}

Lemma~\ref{lemma: min-times convolution} suggests that the $m_{p,k}$ invariants may be nicely packaged into a tropical generating function.
Note that the well-ordering principle implies that arbitrary sums (finite or infinite) in $\trop$ are well-defined.
Let $\trop \lb T\rb$ be the semiring of formal power series in the variable $T$ with coefficients in $\trop$.
The additive identity in $\trop \lb T\rb$ is $\sum_{k\geq 0} \infty T^k$ and the multiplicative identity is $1 + \sum_{k\geq 1}\infty T^k$.

\begin{lemma}
\label{lemma: dec is semiring}
    Let $\dec \subseteq \trop\lb T\rb$ be the set of all series $\sum_{k\geq 0} a_k T^k$ such that $(a_k)$ is a weakly decreasing sequence.
    Then $\dec$ forms a semiring with the same operations and additive identity as $\trop\lb T\rb$, but with multiplicative identity $\frac{1}{1 - T} := \sum_{k\geq 0}T^k$.
\end{lemma}

\begin{proof}
    Let $(a_k)$ be an arbitrary sequence in $\trop$.
    If we define $(b_k)$ by
    \[
        \sum_{k\geq 0}b_k T^k := \frac{1}{1 - T}\sum_{k\geq 0} a_kT^k,
    \]
    then expanding the product we have
    \[
        \sum_{k\geq 0}b_k T^k = \sum_{k\geq 0}\Big(\bigoplus_{j=0}^k a_j\Big)T^k
        = \sum_{k\geq 0}\min_{0\leq j \leq k}\{a_j\} T^k
    \]
    The sequence $b_k = \min_{0\leq j \leq k}\{a_j\}$ is weakly decreasing.
    Furthermore, if $a_k$ is already weakly decreasing, then $b_k = a_k$.
    In particular $\big(\frac{1}{1 - T}\big)^2 = \frac{1}{1 - T}$.
    Thus $\frac{1}{1 - T}$ is an idempotent which projects $\trop\lb T \rb$ onto $\dec$, which implies that $\dec$ is precisely all multiples of $\frac{1}{1 - T}$ in $\trop \lb T \rb$.
    Hence $\dec$ is closed under addition and multiplication; $\dec$ contains the zero element $\frac{\infty}{1 - T}$ by definition.
\end{proof}

Let $S_p : \ideal(\NN) \to \dec$ be the function defined by
\[
     S_p(A) :=  \sum_{k\geq 0} m_{p,k}(A) T^k.
\]
If $A \neq \langle 0\rangle$, then we have the alternative formula
\begin{equation}
\label{eqn: S_p alt}
    S_p(A) = \frac{1}{1 - T}\sum_{a \in A} aT^{v_p(a)}.
\end{equation}

\begin{lemma}
\label{lemma: S_p hom}
    If $p$ is a prime, then $S_p$ is a semiring homomorphism.
\end{lemma}

\begin{proof}
    Let $A, B \in \ideal(\NN)$.
    The ultrametric inequality $v_p(a + b) \geq \min\{v_p(a), v_p(b)\}$ and $a + b \geq a, b$ imply that the min $m_{p,k}(A + B)$ is achieved by an element of $A$ or $B$, hence that $m_{p,k}(A + B) = \min\{m_{p,k}(A), m_{p,k}(B)\}$.
    Hence
    \begin{align*}
        S_p(A + B) &= \sum_{k\geq 0} m_{p,k}(A + B)T^k\\
        &= \sum_{k\geq 0} \min\{m_{p,k}(A), m_{p,k}(B)\}T^k\\
        &= \sum_{k\geq 0} (m_{p,k}(A) \oplus m_{p,k}(B))T^k\\
        &= S_p(A) \oplus S_p(B).
    \end{align*}
    If $A$ or $B$ is the zero ideal, then so is $AB$ and thus $S_p(AB) = S_p(A)S_p(B)$.
    Now suppose that neither $A$ nor $B$ is zero, hence neither is $AB$.
    As we argued in Lemma~\ref{lemma: min-times convolution}, the minimum defining $m_{p,k}(AB)$ is achieved by a product $ab$.
    Therefore
    \begin{align*}
        S_p(AB) &= \frac{1}{1 - T}\sum_{\substack{a \in A\\b\in B}} abT^{v_p(ab)}\\
        &= \frac{1}{1 - T}\sum_{\substack{a \in A\\b\in B}} abT^{v_p(a) + v_p(b)}\\
        &= \Big(\frac{1}{1 - T}\Big)^2\sum_{a \in A}aT^{v_p(a)}\sum_{b\in B}bT^{v_p(b)}\\
        &= S_p(A)S_p(B).
    \end{align*}
    Clearly $S_p(\langle 0 \rangle) = \frac{\infty}{1 - T}$ and $S_p(\langle 1 \rangle) = \frac{1}{1 - T}$.
    Thus $S_p$ is a semiring homomorphism.
\end{proof}

\subsection{Irreducibility Criteria}

The $m_p$ invariants provide a simple means by which to test the irreducibility of ideals and, even better, similarity classes of ideals.

\begin{lemma}
\label{lemma: irred criteria}
    Let $A \subseteq \NN$ be an ideal and let $p \in \NN$ be prime or zero.
    If $m_p(A)$ is prime, then $A$ is irreducible in $\ideal(\NN)$ and $[A]$ is irreducible in $[\ideal(\NN)]$.
\end{lemma}

\begin{proof}
    If $A = BC$, then $[A] = [B][C]$, hence it suffices to prove that $m_p(A)$ being prime implies $[A]$ is irreducible.
    If $[A] = [B][C]$, then $m_p(A) = m_p(B)m_p(C)$.
    Note that $m_p(B) = 1$ if and only if $B = \langle 1 \rangle$.
    Thus if $[B]$ and $[C]$ are nontrivial classes, then $m_p(A)$ is composite.
    This is the contrapositive of our claim.
\end{proof}

\begin{remark}
    The notion of an \emph{irreducible numerical semigroup} has other meanings in the literature.
    For instance, in \cite{rosales2009numerical} they define an irreducible numerical semigroup to be one that is not the intersection of two numerical semigroups which properly contain it.
    We will not be making any further reference to that notion of irreducibility with respect to intersections, and hence for this work, we use irreducibility to mean irreducibility with respect to multiplication of ideals.
\end{remark}

The next result provides three more irreducibility tests, none of which apply to similarity classes.

\begin{proposition}
\label{prop: irred criteria}
    Let $A \subseteq \NN$ be a numerical semigroup.
    If any of the following conditions holds, then $A$ is irreducible.
    \begin{enumerate}
        \item $A$ has two minimal generators.
        \item $A$ has a prime minimal generator.
        \item $m_{p,0}(A)$ and $m_{p,\ell}(A)$ are coprime, where $p$ is prime and $\ell$ is the smallest index such that $m_{p,\ell}(A) < m_{p,0}(A)$.
    \end{enumerate}
\end{proposition}

\begin{proof}
    (1)
    Let $A = \langle m, n\rangle$ with $m < n$ coprime.
    Suppose $A = BC$ where $B = \langle b_1,\ldots, b_r\rangle$ and $C = \langle c_1,\ldots, c_s\rangle$ are minimal generators where $b_1, c_1 > 1$ and the $b_i, c_i$ are strictly increasing.
    Then $m = b_1c_1$.
    The next two smallest generators of $BC$ are $b_1c_2$ and $b_2c_1$.
    Suppose without loss of generality that $b_1c_2 \leq b_2c_1$.
    If $b_1c_2$ is not a minimal generator for $BC$, then it must be a multiple of $b_1c_1$, but that is equivalent to $c_2$ being a multiple of $c_1$.
    This contradicts our assumption that $c_2$ is a minimal generator of $C$.
    Hence $n = b_1c_2$.
    However $b_1 > 1$ divides both $m$ and $n$, which contradicts our assumption that $m$ and $n$ are coprime.
    Thus $A$ is irreducible.

    (2)
    If $B = \langle b_1,\ldots, b_m\rangle$ and $C = \langle c_1,\ldots, c_n\rangle$, then $BC = \langle b_ic_j : 1\leq i \leq m, 1 \leq j \leq n\rangle$.
    The minimal generators for $BC$ are contained in this list of products.
    Thus, if any minimal generator of $A$ is prime, then either $B$ or $C$ must contain 1, which implies that $A$ is irreducible.

    (3)
    Suppose that $A = BC$.
    Then Lemma~\ref{lemma: min-times convolution} implies that for all $k \geq 0$ we have
    \[
        m_{p,k}(A) = \min_{i+j= k}\{ m_{p,i}(B)m_{p,j}(C)\}.
    \]
    If $i'$ and $j'$ are the first indices where $m_{p,i'}(B) < m_{p,0}(B)$ and $m_{p,j'}(C) < m_{p,0}(C)$ and $\ell = \min\{i',j'\}$, then for all $k < \ell$ and $i + j = k$, we must have $m_{p,i}(B) = m_{p,0}(B)$ and $m_{p,j}(C) = m_{p,0}(C)$.
    Thus $m_{p,k}(A) = m_{p,0}(B)m_{p,0}(C) = m_{p,0}(A)$.
    At index $k = \ell$ we get the first decrease in $m_{p,k}(A)$,
    \[
        m_{p,\ell}(A) = \min_{i + j = \ell}\{m_{p,i}(B)m_{p,j}(C)\}
        = \min\{m_{p,\ell}(B)m_{p,0}(C), m_{p,0}(B)m_{p,\ell}(C)\} < m_{p,0}(A).
    \]
    This identity implies that $m_{p,\ell}(A)$ is divisible by either $m_{p,0}(B)$ or $m_{p,0}(C)$, which are proper factors of $m_{p,0}(A)$.
    Thus, taking the contrapositive, if $m_{p,0}(A)$ and $m_{p,\ell}(A)$ are coprime where $\ell$ is the index of the first decrease, then $A$ must be irreducible.
\end{proof}

\subsection{Failure of Unique Factorization}
In this section we show that unique factorization fails for the multiplicative monoids $\ideal(\NN)$ and $[\ideal(\NN)]$.

Let $S = \langle m,n\rangle$ where $m < n$ are coprime.
Proposition~\ref{prop: irred criteria}(1) implies that $S$ is irreducible.

\begin{lemma}
\label{lemma: A power minimal generators}
    The minimal generators of $S^k$ are $m^{i}n^{j}$ where $i, j \geq 0$ and $i + j = k$.
    Hence $S^k$ has embedding dimension $k + 1$.
\end{lemma}

\begin{proof}
    Note that since $m < n$, if $i_1 + j_1 = i_2 + j_2$, then we have $m^{i_1}n^{j_1} < m^{i_2}n^{j_2}$ if and only if $j_1 < j_2$.
    Suppose for the sake of contradiction that $m^in^j$ with $i + j = k$ is the smallest redundant generator.
    Then we must have $j > 0$ since $m^k$ is the smallest element in $S^k$.
    Hence $m^in^j \in S' := \langle m^k, m^{k-1}n,\ldots, m^{i+1}n^{j-1}\rangle$.
    However, note that each of the generators of $S'$ is divisible by $m^{i+1}$, which implies that $m^{i+1}$ divides $m^in^j$.
    But then we must have $m \mid n^j$, which contradicts $m$ and $n$ coprime.
    Therefore none of the $m^in^j$ are redundant.
\end{proof}

Given a finite set $E \subseteq \NN$ with $\max E = k$, let
\[
    S^E := \langle m^{k-j}n^{j} : j \in E\rangle.
\]
Note that $S^E$ is a numerical semigroup if and only if $0 \in E$.
For example, let $[k] := \{0,1,\ldots,k\}$, then $S^{[k]} = S^k$.
If $E_1, E_2 \subseteq \NN$, let $E_1 + E_2 = \{e_1 + e_2 : e_i \in E_i\}$ denote the sumset.
By comparing generators we have
\[
    S^{E_1}S^{E_2} = S^{E_1 + E_2}.
\]
Let $[k]^* := \{0,k\}$.
Then $S^{[k]^*} = \langle m^k, n^k\rangle$ is an irreducible numerical semigroup by Proposition \ref{prop: irred criteria}(1).
Observe that
\[
    [2] + [3]^* = \{0,1,2\} + \{0,3\} = [5].
\]
Hence we have
\begin{equation}
\label{eqn: S1}
    S^2S^{[3]^*} = S^{[2]+ [3]^*} = S^{[5]} = S^5.
\end{equation}
The left hand side of this product has three irreducible factors, and the right side has five.
Thus unique factorization fails for $\ideal(\NN)$.

Another similar example:
\[
    [1] + [2]^* + [2]^* = \{0,1\} + \{0,2\} + \{0,2\} = [5].
\]
Hence we also have
\begin{equation}
\label{eqn: S2}
    S^1 S^{[2]^*}S^{[2]^*} = S^5.
\end{equation}
Digit expansions, including with mixed base, give rise to endless variations on these examples.
This encoding of sumset arithmetic into the multiplicative structure of numerical semigroups demonstrates how flexible the failure of unique factorization is for ideals in $\NN$.

\begin{remark}
    If $A = \langle a_1,\ldots, a_m\rangle$ and $B = \langle b_1,\ldots, b_n\rangle$ are ideals, then the minimal generators of $AB$ are a subset of the products $a_ib_j$.
    Thus we may effectively compute all factorizations of a given ideal.
    We have implemented this factorization algorithm in Python together with all the other invariants and operations discussed in this paper as part of a class of ideals in $\NN$.
    This code is available at \url{https://github.com/tghyde/numerical-ideals}.    
\end{remark}

All the examples of non-unique factorization constructed so far may be resolved by passing to similarity classes.

\begin{lemma}
\label{lemma: S similarity}
    If $E \subseteq \NN$ is a finite set with $\min E = 0$ and $\max E = k$, then $S^E \sim S^k$.
\end{lemma}

\begin{proof}
    Since $S^{[k]^*} \subseteq S^E \subseteq S^k$, we have $\o{S^{[k]^*}} \subseteq \o{S^E} \subseteq \o{S^k}$ by Lemma~\ref{lemma: int closure props}.
    Thus it suffices to prove that each generator of $S^k$ is integral over $S^{[k]^*}$ by Proposition~\ref{prop: noetherian implies similar to int closure}.
    Let $m^in^j$ with $i + j = k$ be a generator of $S^k$.
    Observe that $(m^in^j)^k = m^{ki}n^{kj} \in (S^{[k]^*})^k$.
    Hence Proposition~\ref{prop: power certificate} implies that $m^in^j$ is integral over $S^{[k]^*}$.
\end{proof}

Thus Lemma \ref{lemma: S similarity} implies that the factorizations in \eqref{eqn: S1} and \eqref{eqn: S2} both resolve to $[S]^5$ in $[\ideal(\NN)]$.
This leaves open the possibility that similarity classes factor uniquely into irreducibles.
However, the next theorem shows that unique factorization also fails in $[\ideal(\NN)]$.
We produce an infinite family of counterexamples to unique factorization and another family of examples to show that the number of irreducible factors in a factorization is also not an invariant of a similarity class.

\begin{theorem}
\label{thm: UF fails}
    The semigroup $[\ideal(\NN)]$ of similarity classes of ideals in $\NN$ does not have unique factorization.
    More precisely,
    \begin{enumerate}
        \item If $a < b$ are primes and $c \geq 1$ such that $b < ac$ and $b \nmid c$, let $A = \langle a, b\rangle$ and $B = \langle b, ac\rangle$ be numerical semigroups.
        Suppose $d \in AB$ is any prime such that $d > a$.
        Let $C = \langle ab, a^2c, d\rangle$ and $D = \langle a^2, ab, d\rangle$.
        Then $[A], [B], [C], [D]$ are irreducible, distinct similarity classes and
        \[
            AC = BD.
        \]

        \item The number of irreducible factors is not invariant.
        Let $k\geq 2$ and define numerical semigroups $A := \langle 2, 3\rangle$ and $B := \langle 2, 3^k\rangle$.
        Let $d \in A^k B$ be any odd prime.
        Define $C := \langle 4, 2\cdot 3^k, d\rangle$ and $D:= \langle 2^{k+1},2^k\cdot 3,\ldots, 2\cdot 3^k, d\rangle$.
        Then $[A], [B], [C], [D]$ are irreducible, distinct similarity classes and
        \[
            A^k C = BD.
        \]
    \end{enumerate}
\end{theorem}

\begin{proof}
    (1)
    The numerical semigroups $A$ and $B$ have prime multiplicities, hence $[A]$ and $[B]$ are irreducible.
    Note that $a$ is a prime distinct from $d$.
    Hence $m_a(C) = m_a(D) = d$ and Lemma~\ref{lemma: irred criteria} implies that $[C]$ and $[D]$ are irreducible.
    The ideals $A, B, C, D$ have distinct multiplicities, hence generate pairwise distinct similarity classes.
    Observe that 
    \[
        C = \langle ab, a^2c, d\rangle =  aB + \langle d\rangle\hspace{.5in} D = \langle a^2, ab, d\rangle = aA + \langle d\rangle.
    \]
    Thus
    \[
        AC = aAB + dA\hspace{.5in}BD = aAB + dB.
    \]
    We have $dA = \langle ad, bd\rangle$ and $dB = \langle bd, acd\rangle$.
    Since $d \in AB$ by assumption, $ad, acd \in aAB$.
    Therefore
    \[
        AC = aAB + dA = aAB + \langle bd\rangle = aAB + dB = BD. 
    \]

    (2)
    The argument that $[A], [B], [C], [D]$ are irreducible is the same as in part (1).
    The similarity classes are distinguished by considering $m_0$ and $m_2$ for each ideal.
    The proof of equality is also essentially the same:
    Observe that $C = 2B + \langle d \rangle$ and $D = 2A^k + \langle d \rangle$.
    Thus
    \[
        A^k C = 2A^k B + dA^k \hspace{.5in} BD = 2A^kB + dB.
    \]
    Since $d \in A^k B$ by assumption, we have
    \[
        A^k C = 2A^k B + \langle 3^k d\rangle = BD.
    \]
    Thus one similarity class can have an irreducible factorization with 2 terms or $k + 1$ terms.
\end{proof}

Ideals were introduced to arithmetic as a means of correcting the failure of unique factorization in extensions of the integers.
It is ironic that ideals bring non-unique factorization to the natural numbers, the genesis of unique factorization.
Could there be some extension or refinement of similarity classes that resolves the distinct factorizations highlighted in Theorem~\ref{thm: UF fails}?
It seems unlikely to us, based on the nature of these examples, but we would be happy to be proven wrong.

\section{Invariant Asymptotics}
\label{sec: asymptotics}

In this section we study the asymptotic growth rates of the invariants Frobenius number, genus, and type along sequences of numerical semigroups of the form $(A^k)$ where $A$ is fixed.
The asymptotic formula for Frobenius numbers (Theorem~\ref{thm: frob asymptotic}) is a critical component in the proof of Theorem~\ref{thm: integral closure}.

Numerical semigroups with two generators $S := \langle m, n\rangle$ play a special role in the theory.
These are the simplest numerical semigroups and many of their properties and invariants can be worked out explicitly (whereas such formulas for numerical semigroups with more than two minimal generators are generally unavailable).
Powers of $S$ have simple generating sets,
\[
    S^k = \langle m^k, m^{k-1}n,\ldots, mn^{k-1}, n^k\rangle.
\]
Ong and Ponomarenko \cite{ong2008frobenius} initiated the study of these numerical semigroups by deriving an explicit formula for their Frobenius number.
Tripathi \cite{tripathi2008frobenius} provided an alternative derivation of the Frobenius number as well as a formula for the genus of $S^k$.
See Theorem~\ref{thm: two-gen powers} below.

The \textbf{Ap\'ery set} of a numerical semigroup $A$ with respect to an integer $m$ is
\[
    \ap_m(A) := \{a \in A : a - m \notin A\}.
\]
Equivalently, $\ap_m(A)$ consists of the smallest elements in each congruence class modulo $m$ (which exist by Proposition~\ref{prop: mcnugget}).
Typically we consider the Ap\'ery set of $A$ with respect to the multiplicity $m = m_0(A)$, and we denote that case simply by $\ap(A)$ and call it the Ap\'ery set of $A$.
The Ap\'ery set is a fundamental invariant of a numerical semigroup, essential for computing other invariants and testing membership in $A$.

Kiers, O'Neil, and Ponomarenko \cite{kiers2016numerical} later generalized the results of Ong, Ponomarenko, and Tripathi to a larger family of numerical semigroups generated by a \emph{compound sequence}.
They calculate Ap\'ery sets for all such numerical semigroups.
We recall these results below for later use.

\begin{theorem}
\label{thm: two-gen powers}
    If $S = \langle m, n\rangle$ with $m < n$ coprime, then for $k \geq 1$ we have
    \begin{enumerate}
        \item \cite{ong2008frobenius} $F(S^k) = \frac{(m-1)n^{k+1} - (n-1)m^{k+1}}{n -m} = (m-1)(m^{k-1}n + \ldots + m n^{k-1} + n^k) - m^k$,

        \item \cite{tripathi2008frobenius} $g(S^k) = \frac{F(S^k) + 1}{2}$,

        \item \cite[Thm. 15]{kiers2016numerical} $\ap(S^k) = \{j_1m^{k-1}n+j_2m^{k-2}n^2+\ldots+j_kn^k:0\leq j_i < m\}$.
    \end{enumerate}
\end{theorem}

\subsection{Symmetry and Type}
\label{sec: symmetry}
If $A$ is a numerical semigroup, let $\chi_A$ denote its characteristic function.
Proposition~\ref{prop: mcnugget} implies that the sequence $(\chi_A(n))_{n\in \NN}$ is eventually all 1's.
A binary sequence which is eventually constant equal to 1 encodes a partition:
Starting at $(0,g(A))$ on the $y$-axis, interpret each 0 as a step down, and each 1 as a step to the right.
The resulting path, together with the positive $x$- and $y$-axes, bounds a Young diagram of a partition.

\begin{example}
\label{ex: symm}
    Let $A = \langle 3, 7\rangle$.
    The characteristic sequence of $A$ is
    \[
        100100110110111\ldots
    \]
    and the corresponding Young diagram is
    \begin{center}
        \includegraphics[scale=.35]{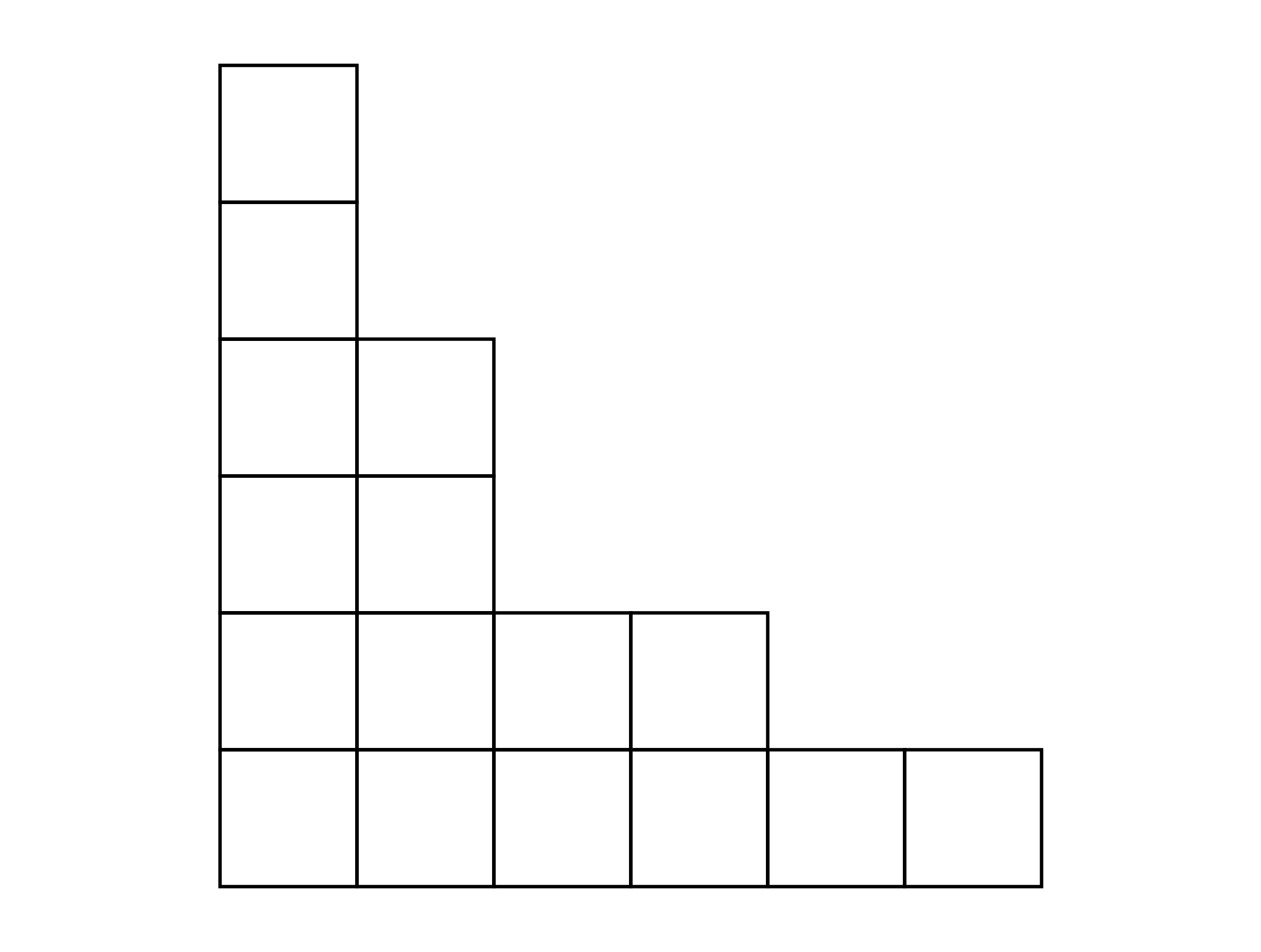}
    \end{center}
\end{example}

Note that the partition in Example~\ref{ex: symm} is symmetric with respect to reflection across the diagonal.
A \textbf{symmetric} numerical semigroup is one for which its corresponding partition is symmetric.
Thus Example~\ref{ex: symm} shows that $\langle 3, 7\rangle$ is symmetric.
Furthermore, it is known that every two-generator numerical semigroup is symmetric (see, e.g. \cite[Cor. 4.7]{rosales2009numerical}).
Symmetric numerical semigroups may be characterized by the relationship between their genus and Frobenius numbers: $A$ is symmetric if and only if $g(A) = \frac{F(A) + 1}{2}$ (see \cite[Cor. 4.5]{rosales2009numerical}).

\begin{corollary}
\label{cor: powers are symmetric}
    If $S = \langle m, n\rangle$ where $m < n$ are coprime, then $S^k$ is symmetric for all $k\geq 0$.
\end{corollary}

\begin{proof}
    This follows directly from the result of Tripathi, Theorem~\ref{thm: two-gen powers}(2).
\end{proof}

Moving beyond two-generator numerical semigroups, the situation is more complicated.
There are numerical semigroups $A$ that are symmetric, but which have non-symmetric powers.
For example, $A = \langle 4, 6, 7\rangle$ is symmetric, but $A^2$ is not symmetric.
There are also examples where $A$ is not symmetric but $A^2$ is symmetric, though these are harder to find.
For example, $A = \langle 12, 18, 20, 39, 67\rangle$ is not symmetric, but $A^2$ is symmetric.

We can refine this analysis by consider how \emph{type} behaves under taking powers.
Let $A \subseteq \NN$ be a numerical semigroup.
Given $a, b \in \NN$ we say $a \leq_A b$ if $b - a \in A$.
The relation $\leq_A$ defines a partial order on $\NN$.
The \textbf{pseudo-Frobenius numbers} of $A$ are the maximal elements of $\NN\setminus A$ with respect to $\leq_A$.
Let $\pf(A)$ denote the set of all pseudo-Frobenius numbers.
The \textbf{type} of $A$ is $\t(A) := |\pf(A)|$.
If $m \in \NN$, let $\maxap_m(A)$ denote the set of maximal elements of $\ap_m(A)$ with respect to $\leq_A$.
A numerical semigroup $A$ is symmetric if and only if $\t(A) = 1$ (see \cite[Prop. 2.19]{rosales2009numerical}).

\begin{question}
\label{question: type growth}
    Let $A \subseteq \NN$ be a numerical semigroup.
    Does $\lim_{k\to\infty} \log \t(A^k)/k$ exist?
    If so, how do we express this limit in terms of other invariants of $A$?
\end{question}

Note that $\t(A) \leq m_0(A) - 1$ for all numerical semigroups $A$ (see \cite[Cor. 2.23]{rosales2009numerical}).
Hence $\t(A^k) \leq m_0(A)^k - 1$, which implies that
\[
    \limsup_{k\to\infty} \frac{\log \t(A^k)}{k} \leq m_0(A).
\]
However Corollary \ref{cor: powers are symmetric} implies this bound is not sharp.
The next proposition allows us to show that $\t(A^k)$ grows monotonically in many examples; it generalizes a result due to Nari \cite[Prop. 6.6]{Nari2013}.

\begin{proposition}
\label{prop: gluing type}
    Let $A, B \subseteq \NN$ be numerical semigroups and let $m, n \in \NN$ be coprime natural numbers such that $n \in A$.
    Then
    \[
        \pf(mA + nB) = m\pf(A) + n\maxap_m(B),
    \]
    and
    \[
        \t(mA + nB) = \t(A)|\maxap_m(B)|.
    \]
\end{proposition}

\begin{proof}
    Let $C := mA + nB$.
    Given $b \in \ap_m(B)$, let $C_b := mA + nb$.
    Since $m$ and $n$ are coprime we have
    \begin{equation}
    \label{eqn: C decomp}
        C = \bigsqcup_{b \in \ap_m(B)} C_b.
    \end{equation}
    Suppose $a \in \pf(A)$ and $b \in \maxap_m(B)$ and define
    \[
        c := ma + nb.
    \]
    Note that $a \notin A$ implies that $c \notin C$.
    Let $d \in C$ be a nonzero element.
    The decomposition \eqref{eqn: C decomp} implies that $d = ma' + nb'$ for some $a' \in A$ and some $b' \in \ap_m(B)$.
    The definition of $\ap_m(B)$ tells us that
    \[
        b + b' = b'' + me
    \]
    for some $b'' \in \ap_m(B)$ and $e \in \NN$.
    Thus
    \begin{align*}
        c + d &= (ma + nb) + (ma' + nb')\\
        &= m(a + a') + n(b + b')\\
        &= m(a + a') + n(b'' + me)\\
        &= m(a + a' + ne) + nb''.
    \end{align*}
    Since $a \in \pf(A)$ and $a', n \in A$ we have $a + a' + ne \in A$ if and only if $a' + ne \neq 0$.
    If $a' + ne = 0$, then $a' = e = 0$.
    Hence $b + b' = b''$, which gives us $b \leq_B b''$.
    Our assumption that $b \in \maxap_m(B)$ then implies that $b = b''$, which is to say that $b' = 0$.
    But then $d = ma' + nb' = 0$, contradicting our assumption that $d$ was nonzero.
    Hence $c \in \pf(C)$.
    Therefore $m\pf(A) + n\maxap_m(B) \subseteq \pf(C)$.

    Let $c \in \pf(C)$ and define $b$ to be the unique element of $\ap_m(B)$ such that $c \equiv nb \bmod m$.
    Hence $c = ma + nb$ for some integer $a$.
    Since $c \notin C$ we must have $a \notin A$ by \eqref{eqn: C decomp}.
    For every positive $a' \in A$ we have $c + ma' \in C$ by definition of pseudo-Frobenius elements.
    Hence
    \[
        c + ma' = (ma + nb) + ma' = m(a + a') + nb \in C
    \]
    and \eqref{eqn: C decomp} implies that $a + a' \in A$.
    Therefore $a \in \NN$ and in fact $a \in \pf(A)$.
    Suppose $b \notin \maxap_m(B)$.
    Then we can write $b + b' = b''$ where $b'' \in \ap_m(B)$ and $b' \in B$ is positive.
    Hence $nb' \in C$ is positive and thus $c + nb' \in C$.
    However
    \[
        c + nb' = (ma + nb) + nb' = ma + n(b + b') = ma + nb'' \in C,
    \]
    and \eqref{eqn: C decomp} implies that $a \in A$, a contradiction.
    Therefore $b \in \maxap_m(B)$.
    Hence $\pf(C) \subseteq m\pf(A) + n\maxap_m(B)$, and we conclude that the sets are equal.
    Since $m$ and $n$ are coprime, the decomposition
    \[
        \pf(C) = m\pf(A) + n\maxap_m(B)
    \]
    gives us a bijection $\pf(C) \cong \pf(A) \times \maxap_m(B)$.
    Thus
    \[
        \t(C) = \t(A)|\maxap_m(B)|.\qedhere
    \]
\end{proof}

\begin{remark}
\label{remark: Nari}
    If $m \in B$, then $\maxap_m(B) = \pf(B) + m$ and $|\maxap_m(B)| = \t(B)$.
    Thus we get 
    \[
        \t(mA + nB) = \t(A)\t(B)
    \]
    in that case, which is precisely the conclusion of Nari's result.
    Our result does not require $m \in B$, but gives a similar conclusion.
\end{remark}

\begin{example}
    Proposition~\ref{prop: gluing type} gives an alternative proof that for $A := \langle m, n\rangle$ the powers $A^k$ is symmetric for all $k\geq 1$.
    Note that $A^{k+1} = \langle m,n\rangle A^k = mA^k + n^kA$, $m$ and $n^k$ are coprime, and $n^k \in A^k$.
    Hence Proposition~\ref{prop: gluing type} implies that $\t(A^{k+1}) = \t(A^k)|\maxap_{m}(A)|$.
    Since $m \in A$, Remark~\ref{remark: Nari} gives $|\maxap_m(A)| = \t(A) = 1$, and thus by induction that $\t(A^k) = \t(A) = 1$ for all $k\geq 1$.
    Hence $A^k$ is symmetric for all $k \geq 1$.
\end{example}

The next proposition allows us to construct examples where we can answer Question~\ref{question: type growth}.

\begin{proposition}
\label{prop: quadratic reduction}
    Let $m \geq 1$ and let $B$ be a numerical semigroup with $n \in B$ coprime to $m$.
    If $A := \langle m\rangle + B$ and $B^2 \subseteq mA + nB$, then $\t(A^{k}) = \t(A)|\maxap_m(B)|^{k-1}$ for all $k\geq 1$.
    In particular, 
    \[
        \lim_{k\to\infty} \frac{\log \t(A^k)}{k} = \log |\maxap_m(B)|.
    \]
\end{proposition}

\begin{proof}
    First observe that
    \[
        A^{k+1} = (\langle m\rangle + B)A^k = mA^k + B^{k+1}.
    \]
    Our assumption $B^2 \subseteq mA + nB$ allows us to make a simple inductive argument that for all $k\geq 1$
    \[
        B^{k+1} \subseteq mA^k + n^k B,
    \]
    hence
    \[
        A^{k+1} = mA^k + n^k B.
    \]
    Since $n^k \in A^k$, Proposition~\ref{prop: gluing type} implies that
    \[
        \t(A^{k+1}) = \t(A^k)|\maxap_m(B)|.
    \]
    Thus a simple induction implies that $\t(A^{k}) = \t(A)|\maxap_m(B)|^{k-1}$ for all $k\geq 1$.
    Hence
    \[
        \lim_{k\to\infty}\frac{\log \t(A^k)}{k} 
        = \log|\maxap_m(B)|.\qedhere
    \]
    
\end{proof}

\begin{example}
\label{ex: exp type growth}
    Let $A := \langle 4, 7, 9\rangle = \langle 4\rangle + B$ where $B := \langle 7,9\rangle$, and let $n = 7$.
    Then $B^2 = \langle 49, 63, 81\rangle$.
    Clearly $49, 63 \in 7B$ and $81 = 4\cdot 8 + 7\cdot 7 \in 4A + 7B$, hence $B^2 \subseteq 4A + 7B$.
    Therefore Proposition~\ref{prop: quadratic reduction} implies that
    \[
        \t(A^{k+1}) = \t(A^k)|\maxap_4(B)|.
    \]
    Note that $\ap_4(B) = \{0,7,9,14\}$, hence $\maxap_4(B) = \{9,14\}$.
    Thus by induction and $\t(A) = 2$ it follows that $\t(A^k) = 2^k$.
\end{example}

We close this section with one final application of Proposition~\ref{prop: gluing type}.

\begin{proposition}
    Let $1 < m < n_1 < n_2$ be coprime integers.
    If $n_1 \equiv n_2 \bmod m$, then the product $\langle m, n_1\rangle\langle m, n_2\rangle$ is symmetric.
\end{proposition}

\begin{proof}
    Let $A := \langle m, n_1\rangle$, $B := \langle m, n_2\rangle$, and $C := AB$.
    By definition we have
    \begin{equation}
    \label{eqn: 2 gen product}
        C= \langle m, n_1\rangle\langle m, n_2\rangle = mA + n_2A.
    \end{equation}
    Since $n_1 < n_2$ and $n_1 \equiv n_2 \bmod m$, we have $n_2 \in A$.
    Thus Proposition~\ref{prop: gluing type} and Remark~\ref{remark: Nari} imply that $\t(C) = \t(A) |\maxap_m(A)| = \t(A)^2 = 1$.
    Therefore $C$ is symmetric.
\end{proof}

\subsection{Frobenius and Genus}
\label{sec: frob and gen}

Suppose $A$ is a numerical semigroup and let 
\[
    \mmax(A) := \max_{p} m_p(A).
\]
Since $m_p(A) = m_0(A)$ for all primes $p \nmid m_0(A)$, this maximum is well-defined.
The invariant $\mmax(A)$ controls the growth rate of genera and Frobenius numbers of powers of $A$.

\begin{theorem}
\label{thm: frob asymptotic}
    Let $A$ be a numerical semigroup.
    Then
    \[
        \lim_{k\to\infty}\frac{\log g(A^k)}{k} = \lim_{k\to\infty} \frac{\log F(A^k)}{k} =  \log \mmax(A).
    \]
\end{theorem}

\begin{proof}
    First we consider $g(A^k)$.
    Let $m := m_0(A)$ and $n := \mmax(A)$.
    Let $p$ be a prime such that $n = m_p(A)$.
    Then $m_p(A^k) = n^k$ and any natural number $j \leq n^k$ which is not divisible by $p$ is not an element of $A^k$.
    Thus
    \[
        g(A^k) \geq \Big(1 - \frac{1}{p}\Big)n^k,
    \]
    which implies that
    \begin{equation}
    \label{eqn: liminf ineq}
        \liminf_{k\to\infty}\frac{\log g(A^k)}{k} \geq \log n.
    \end{equation}

    Let $m = p_1^{e_1}\cdots p_s^{e_s}$ be the prime factorization of $m_0(A)$, and let $n_i := m_{p_i}(A)$.
    Let $d_i := \gcd(m, n_i)$ and define $A_i := \langle m, n_i\rangle$.
    Note that $p_i \nmid d_i$.
    Setting $m = d_i m_i'$ and $n_i = d_i n_i'$ we have $A_i = d_i A_i'$, where $A_i' := \langle m_i', n_i'\rangle$ is a numerical semigroup.
    
    For each $i$, let
    \(
        W_i := d_i^k\ap({A_i'}^k).
    \)
    Since \(A_i=d_iA_i'\), we have
    \(
        W_i \subseteq A_i^k \subseteq A^k.
    \)
    The elements of $\ap({A_i'}^k)$ represent every residue class modulo ${m_i'}^k$, hence $W_i$ represents all the elements of the subgroup
    \[
        H_i := d_i^k\ZZ/\langle m^k \rangle
        \subseteq \ZZ/\langle m^k\rangle.
    \]
    Since $p_i\nmid d_i$ for each $i$ we have
    \(
        \gcd(d_1^k,\dots,d_s^k,m^k)=1.
    \)
    Thus
    \begin{equation}
    \label{eqn: H sum}
        H_1+\cdots+H_s=\ZZ/\langle m^k\rangle.
    \end{equation}
    If 
    \[
        W := W_1+\cdots+W_s
        =
        \{w_1+\cdots+w_s : w_i\in W_i\} \subseteq A^k,
    \]
    then \eqref{eqn: H sum} implies that $W$ contains representatives of every congruence class modulo $m^k$.
    Hence
    \[
        \max_{w\in \ap(A^k)} w
        \leq \max W
        \leq \sum_{i=1}^s \max W_i.
    \]
    For each $i$, Theorem~\ref{thm: two-gen powers}(1) implies that
    there is a positive constant $c_i$, independent of $k$, such that
    \[
        \max \ap({A_i'}^k)
        =
        F({A_i'}^k)+{m_i'}^k
        \leq c_i{n_i'}^k.
    \]
    Therefore
    \[
        \max W_i
        =
        d_i^k\max \ap({A_i'}^k)
        \leq c_i d_i^k{n_i'}^k
        =
        c_i n_i^k
        \leq c_i n^k.
    \]
    Setting $c:=\sum_{i=1}^s c_i$, we get
    \[
        \max_{w\in \ap(A^k)} w
        \leq cn^k.
    \]
    Selmer's formula \cite[Prop. 2.12]{rosales2009numerical} now gives
    \[
        g(A^k)
        =
        \frac{1}{m^k}
        \sum_{w\in\ap(A^k)} w
        -
        \frac{m^k-1}{2}
        \leq
        cn^k.
    \]
    Hence
    \[
        \limsup_{k\to\infty}\frac{\log g(A^k)}{k}
        \leq \log n.
    \]
    Together with \eqref{eqn: liminf ineq} this proves
    \[
        \lim_{k\to\infty} \frac{\log g(A^k)}{k} = \log n.
    \]
    Next we turn to $F(A^k)$.
    Let $S(A)$ denote the number of $a \in A$ such that $a < F(A)$.
    Then $F(A) = g(A) + S(A) - 1$.
    Note that if $a \in A$ is smaller than $F(A)$, then $F(A) - a$ must be an element of the gap set.
    Hence $S(A) \leq g(A)$, which implies $F(A) \leq 2g(A) - 1$.
    Thus
    \[
        \limsup_{k\to\infty} \frac{\log F(A^k)}{k} \leq \lim_{k\to \infty} \frac{\log g(A^k)}{k} = \log n.
    \]
    For the other direction, note that $g(A) \leq F(A)$ for any numerical semigroup $A$.
    Hence
    \[
        \liminf_{k\to\infty} \frac{\log F(A^k)}{k} \geq \lim_{k\to \infty} \frac{\log g(A^k)}{k} = \log n.
    \]
    Thus
    \[
        \lim_{k\to\infty} \frac{\log F(A^k)}{k} = \log n.\qedhere
    \]
\end{proof}

\section{Integral Closure of Ideals in $\NN$}
\label{sec: integral closure N}

In this section we provide a simple and efficient characterization of the integral closure of an ideal $A \subseteq \NN$.
We begin with the following corollary of Theorem~\ref{thm: frob asymptotic}.

\begin{corollary}
\label{cor: mmax criteria}
    Let $A \subseteq \NN$ be a numerical semigroup.
    If $a \geq \mmax(A)$, then for all sufficiently large $k \geq 1$ we have $a^k \in A^k$.
    In particular, $a \geq \mmax(A)$ implies $a \in \o A$.
\end{corollary}

\begin{proof}
    Since $a = \mmax(A) \in A$, that case is immediate.
    Suppose that $a > \mmax(A)$.
    Then Theorem~\ref{thm: frob asymptotic} implies that
    \[
        \lim_{k\to\infty} \frac{\log F(A^k)}{k} = \log \mmax(A) < \log a,
    \]
    hence for all sufficiently large $k$ we have
    \[
        F(A^k) < a^k.
    \]
    Thus by definition of the Frobenius number, $a^k \in A^k$.
    Therefore $a \in \o A$ by Proposition~\ref{prop: power certificate}.
\end{proof}

In Section \ref{sec: multiplicities} we constructed the $S_p$ homomorphisms defined by $S_p(A) = \sum_{k\geq 0} m_{p,k}(A)T^k$.
The semiring $\trop[T]$ is not multiplicatively cancellative, hence the $S_p$ are not well-defined on similarity classes.
Note that evaluating the series $S_p(A)$ at $T = \lambda \geq 1$ yields a nonzero real number for all nonzero ideals $A$.
Let $\rtrop$ denote the tropical min-times semiring of nonnegative real numbers together with $\infty$.
We define the \textbf{$p$-support functions} 
\[
    \Phi_{p,\lambda}(A) := S_p(A)(\lambda) = \min_{k\geq 0}m_{p,k}(A)\lambda^k \in \rtrop,
\]
where $\lambda \geq 1$ is a real parameter.
If $A \neq \langle 0\rangle$, then our alternative formula \eqref{eqn: S_p alt} for $S_p$ translates to
\begin{equation}
\label{eqn: alt formula for Phi}
    \Phi_{p,\lambda}(A) = \min_{a \in A\setminus\{0\}} a\lambda^{v_p(a)}.
\end{equation}

With $p$ and $\lambda$ fixed, Lemma \ref{lemma: S_p hom} implies that $\Phi_{p,\lambda} : \ideal(\NN) \to \rtrop$ is a semiring homomorphism.
If $A \neq 0$, then $\Phi_{p,\lambda}(A)$ is a positive real number.
Hence $\Phi_{p,\lambda}$ is cancellative on nonzero ideals.
The data packaged in the functions $\Phi_{p,\lambda}$ with $p$ fixed and $\lambda$ varying can also be encoded as a Newton polygon.
Given a prime $p$, let 
\[
    N_p(A) := \text{lower convex hull of } \{(k, \log m_{p,k}(A)) : k \geq 0\}.
\]
We call $N_p(A)$ the \textbf{$p$th Newton polygon} of $A$.
Let $V_p(a) := (v_p(a), \log(a))$.

\begin{lemma}
\label{lemma: support function translation}
    Let $p$ be prime, let $A \subseteq \NN$ be an ideal, and let $a \in \NN$.
    Then $V_p(a) \in N_p(A)$ if and only if
    \(
        \Phi_{p,\lambda}(A) \leq a \lambda^{v_p(a)}
    \)
    for all $\lambda \geq 1$.
\end{lemma}

\begin{proof}
    Every point $(x_0, y_0)$ which does not belong to the convex set $N_p(A)$ is separated from it by a line.
    Thus $(x_0,y_0) \in N_p(A)$ if and only if
    \[
        x_0u + y_0 \geq \min_{k\geq 0}\{ ku + \log m_{p,k}(A)\}
    \]
    for all $u \geq 0$.
    Substituting $u = \log \lambda$, this becomes
    \[
        x_0\log \lambda + y_0 \geq \min_{k\geq 0}\{ k\log \lambda + \log m_{p,k}(A)\}
        = \min_{k\geq 0} \log \big(m_{p,k}(A) \lambda^k\big)
        = \log \Phi_{p,\lambda}(A)
    \]
    for all $\lambda \geq 1$.
    Exponentiating we get $(x_0,y_0) \in N_p(A)$ if and only if
    \[
        \Phi_{p,\lambda}(A) \leq e^{y_0} \lambda^{x_0}
    \]
    for all $\lambda \geq 1$.
    Applying this criterion to $V_p(a) = (v_p(a), \log a)$ gives us the result.
\end{proof}

Given convex sets $X$ and $Y$, let 
\(
    X + Y = \{x + y : x \in X, y \in Y\}
\)
denote their Minkowski sum.
Let $\C$ denote the semigroup of convex subsets of $\RR^2$ with respect to Minkowski sum.

\begin{proposition}
\label{prop: Newton polygon}
    Given a prime $p$, the map $N_p : \ideal(\NN) \to \C$ is a semigroup homomorphism factoring through the universal cancellative quotient $A \mapsto [A]$.
    That is, $A \sim B$ implies that $N_p(A) = N_p(B)$.
\end{proposition}

\begin{proof}
    Lemma~\ref{lemma: min-times convolution} implies that
    \[
        \log m_{p,k}(AB) = \min_{i+j = k}\{\log m_{p,i}(A) + \log m_{p,j}(B)\},
    \]
    which is precisely the expression for the vertices of the Minkowski sum of $N_p(A)$ and $N_p(B)$.
    Hence
    \[
        N_p(AB) = N_p(A) + N_p(B).
    \]
    Suppose $C \neq \langle 0 \rangle$ and $AC = BC$.
    Then $N_p(AC) = N_p(BC)$ and Lemma~\ref{lemma: support function translation} implies that 
    \[
        \Phi_{p,\lambda}(A)\Phi_{p,\lambda}(C) = \Phi_{p,\lambda}(AC) = \Phi_{p,\lambda}(BC) = \Phi_{p,\lambda}(B)\Phi_{p,\lambda}(C)    
    \]
    for all $\lambda \geq 1$.
    Then $C$ nonzero implies that $\Phi_{p,\lambda}(A) = \Phi_{p,\lambda}(B)$ for all $\lambda \geq 1$, which in turn implies $N_p(A) = N_p(B)$.
    Hence $N_p$ factors through $[\ideal(\NN)]$.
\end{proof}

We now present the main result of this section, an effective characterization of integral closures of ideals in $\NN$.

\begin{theorem}
\label{thm: integral closure}
    Let $A \subseteq \NN$ be an ideal and let $a \in \NN$.
    Then the following are equivalent
    \begin{enumerate}
        \item $a^k \in A^k$ for some $k\geq 1$,
        \item $a \in \o A$,
        \item $V_p(a) \in N_p(A)$ for all primes $p$.
    \end{enumerate}
\end{theorem}

\begin{proof}
    Proposition~\ref{prop: power certificate} implies (1) $\Rightarrow$ (2).
    Proposition \ref{prop: Newton polygon} and Proposition \ref{prop: noetherian implies similar to int closure} imply that $N_p(A) = N_p(\o A)$.
    If $a \in \o A$, then $V_p(a) \in N_p(\o A) = N_p(A)$ for all primes $p$, which proves (2) $\Rightarrow$ (3).
    Hence it suffices to prove (3) $\Rightarrow$ (1).
    Note that if $A = \langle 0\rangle$, then the assertion is immediate.
    
    So suppose that $A$ is a nonzero ideal and that $p$ is a prime.
    By assumption we have $V_p(a) \in N_p(A)$.
    If $v_p(a) \geq v_p(m_0(A))$ and $m_0(A) \leq a$, let $e_p = 1$ and $b_p = m_0(A)$.
    Otherwise there is some $k_1 < k_2$ such that the point $(v_p(a), \log a)$ lies above the line segment connecting $(k_1,\log m_{p,k_1}(A))$ and $(k_2,\log m_{p,k_2}(A))$.
    Let $w_i := m_{p,k_i}(A)$ and note that $k_i = v_p(w_i)$.
    This geometric assertion about $V_p(a)$ translates to
    \[
        \frac{k_2 - v_p(a)}{k_2-k_1} \log w_1 + \frac{v_p(a) - k_1}{k_2 - k_1}\log w_2 \leq \log a,
    \]
    or equivalently
    \[
        b_p := w_1^{k_2-v_p(a)}w_2^{v_p(a) - k_1} \leq a^{k_2 - k_1}.
    \]
    Let $e_p := k_2 - k_1$.
    Finally, observe that $w_1, w_2 \in A$ imply that $b_p \in A^{e_p}$ and
    \[
        v_p(b_p) = (k_2 - v_p(a))k_1 + (v_p(a) - k_1)k_2 = e_pv_p(a).
    \]
    If $p\nmid m_0(A)$, then $V_p(a) \in N_p(A)$ is equivalent to $m_0(A) \leq a$.
    Since $0 = v_p(m_0(A)) \leq v_p(a)$, we can let $e_p = 1$ and $b_p = m_0(A)$.

    Thus for each prime $p$ there is an element $b_p$ and an integer $e_p \geq 1$ such that $b_p \leq a^{e_p}$ and $v_p(b_p) \leq e_pv_p(a)$.
    Since $e_p = 1$ for all but finitely many primes $p$, $e := \lcm_p e_p$ exists.
    Replacing $b_p$ with $b_p^{e/e_p}$ we have $b_p \leq a^e$ and $v_p(b_p) \leq e v_p(a)$ for all primes $p$.
    Let $g := \gcd_p b_p$ and write $b_p = g b_p'$.
    Then $v_p(g) \leq e v_p(a)$ for all primes $p$, which is to say that $a^e = ga'$ for some $a' \in \NN$.
    Consider the ideal $B := \langle b_p : p \text{ prime}\rangle \subseteq A^e$ and the numerical semigroup $B' = \langle b_p' : p\text{ prime}\rangle$.
    Then $B = gB'$.
    Since $b_p' \leq a'$ for all $p$, Lemma \ref{lemma: m_pk among generators} implies that $\mmax(B') \leq a'$.
    Thus Proposition~\ref{prop: power certificate} implies there is some $k \geq 1$ such that ${a'}^k \in {B'}^k$.
    Hence
    \[
        a^{ek} = g^k {a'}^k \in g^k {B'}^k = B^k \subseteq A^{ek}.\qedhere 
    \]
\end{proof}

Theorem~\ref{thm: integral closure} implies that types (1), (2), and (3) integrality are equivalent for ideals in $\NN$.
The characterization of $\o A$ in terms of Newton polygons is analogous to the theorem characterizing the integral closure of a monomial ideal in terms of the Newton polyhedron defined by its exponent vectors \cite[Prop. 1.4.6]{huneke2006integral}.
From this perspective, natural numbers can be thought of as monomials in the primes.
Translating between Newton polygons and the support functions, Theorem~\ref{thm: integral closure} may also be viewed as an analog of Rees's theorem characterizing integral closures of ideals in Noetherian rings (original proof by Rees \cite{rees1956valuations}, see Huneke and Swanson \cite[Thm. 10.2.2]{huneke2006integral} for a contemporary statement).
The functions $\log \Phi_{p,\lambda}$ define tropical valuations on ideals of $\NN$, and Theorem~\ref{thm: integral closure} asserts that finitely many of these valuations determine the integral closure of $A$, namely those corresponding to nontrivial slopes of $N_p(A)$, which only exist for the finitely many primes $p \mid m_0(A)$.

\begin{example}
\label{ex: 8, 11}
    Let $A = \langle 8, 11\rangle$.
    Then every $a \geq 11$ belongs to $\o A$ by Corollary~\ref{cor: mmax criteria} and it remains to decide whether 9 or 10 belong to $\o A$.
    The only relevant prime here is $p = 2$.
    The Newton polygon $N_2(A)$ consists of a single segment and a computation shows that $V_2(9) \notin N_2(A)$ and $V_2(10) \in N_2(A)$.
    Therefore,
    \[
        \o A = \{0, 8, 10, 11, 12,\ldots\}.
    \]
    This example nicely highlights the value of our theoretical analysis of integral closures.
    If we try to use the iterated colon ideal method to find an ideal $B$ such that $10B \subseteq AB$, the convergence happens incredibly slowly.
    The sequence $B_k$ strictly decreases for at least $0 \leq k \leq 20$.
    The ideal $B_{20}$ has 370 minimal generators and multiplicity 512.
    It is unclear how many more iterates are required to reach stabilization; the computations become prohibitively slow.
    Theorem~\ref{thm: integral closure} shows that $10^k \in A^k$ for some sufficiently large $k$.
    Using the characterization of the Ap\'ery set of $\langle 8, 11\rangle^k$ provided by Theorem~\ref{thm: two-gen powers}(3) one may show that the first $k$ for which $10^k \in A^k$ is $k = 198$.
\end{example}

If $m$ is squarefree, the Newton polygon conditions are vacuous.

\begin{corollary}
    Let $S := \langle m, n\rangle$ where $2 \leq m < n$ are coprime.
    If $m$ is squarefree, then
    \[
        \o S = \langle m\rangle + \{k \in \NN : k \geq n\}.
    \]
\end{corollary}

\begin{proof}
    If $m$ is squarefree, then for each prime $p \mid m$, the Newton polygon $N_p(S)$ consists of a single segment of width one.
    Corollary \ref{cor: mmax criteria} implies that every $a \geq n$ belongs to $\o S$, and clearly no $a < m$ belongs to $\o S$.
    Thus if $m < a < n$ and $V_p(a) \in N_p(S)$, then $v_p(a) \geq v_p(m) = 1$.
    Hence if such an $a$ belongs to $\o S$, it follows that $m \mid a$.
    Thus
    \(
        \o S = \langle m\rangle + \{k \in \NN : k \geq n\}.
    \)
\end{proof}

Although $\o {AB} \neq \o A \cdot \o B$ in general, this does hold for scaling.

\begin{proposition}
\label{prop: scale integral closure}
    If $A \subseteq \NN$ is a numerical semigroup and $g \in \NN$, then $\o {gA} = g\o A$.    
\end{proposition}

\begin{proof}
    If $g = 0$, then $\o{gA} = \{0\} = g\o A$.
    Now suppose that $g \neq 0$.
    If $a \in \o A$, then there exists a nonzero ideal $B$ such that $aB \subseteq AB$.
    Hence $gaB \subseteq gAB$, and we conclude that $ga \in \o {gA}$.
    Thus $g\o A \subseteq \o {gA}$.
    Conversely, suppose $a \in \o {gA}$.
    Let $B$ be a nonzero ideal such that $aB \subseteq gAB$.
    Lemma~\ref{lemma: N ideal is principal times NS} implies we can express $B = hC$ for some nonzero $h$ and some numerical semigroup $C$.
    Thus
    \[
        ahC \subseteq ghAC.
    \]
    Since $h \neq 0$, we can cancel it from both sides of this containment to get
    \[
        aC \subseteq gAC.
    \]
    The numerical semigroup $C$ contains $n$ and $n+1$ for some sufficiently large $n$.
    The above containment implies that $g$ divides both $an$ and $a(n+1)$.
    Thus $a = gc$ for some $c \in \NN$ and substitution yields
    \[
        gcC \subseteq gAC.
    \]
    Canceling $g$ we see that $c \in \o A$.
    Therefore $a = gc \in g \o A$, completing our proof that $g\o A = \o {gA}$.
\end{proof}

\printbibliography
\end{document}